\documentclass[11pt]{article} 

\usepackage{bbm}
\usepackage{subcaption}
\usepackage{tikz}
\usepackage{placeins}
\usepackage{todonotes}
\usepackage{comment}
\usepackage{amsmath,amssymb,amsthm}
\usepackage{bbm}
\allowdisplaybreaks
\usepackage{nicefrac}
\usepackage[margin=1in]{geometry}

\newtheorem{theorem}{Theorem}
\newtheorem{lemma}{Lemma}
\newtheorem{corollary}{Corollary}
\newtheorem{definition}{Definition}
\newtheorem{proposition}{Proposition}
\newtheorem{conjecture}{Conjecture}
\newtheorem{remark}{Remark}

\usepackage{mathtools}
\usepackage[hidelinks]{hyperref}
\newcommand*{\email}[1]{\href{mailto:#1}{\nolinkurl{#1}} } 

\usepackage{xcolor}

\newcommand\stleq{\stackrel{\mathclap{\normalfont\mbox{\small{st}}}}{\leq}}
\newcommand\stgeq{\stackrel{\mathclap{\normalfont\mbox{\small{st}}}}{\geq}}

\newcommand{\f}[1]{\exp\left(\frac{\beta}{n}{#1}\right)}

\title{Attracting Random Walks} 

\author{%
 Julia~Gaudio\footnote{Massachusetts Institute of Technology, United States of America.
    \email{jgaudio@mit.edu}.}
 \and Yury~Polyanskiy\footnote{Massachusetts Institute of Technology, United States of America.
    \email{yp@mit.edu}. Some key suggestions for the proof ideas in this paper came from David Gamarnik, Patrick Jaillet, Eyal Lubetzky, Reza Gheissari, and Yuval Peres, detailed in the acknowledgements. The authors are extremely grateful to these people for the guidance and help the during the progression of this project.}
}

\begin{document}
\maketitle

\begin{center}
\textbf{Abstract}
\end{center}
This paper introduces the Attracting Random Walks model, which describes the dynamics of a system of particles on a
graph with $n$ vertices. At each step, a single particle moves to an adjacent vertex (or stays at the current one) 
with probability proportional to the exponent of the number of other particles at a vertex. From an applied standpoint,
the model captures the \textit{rich get richer} phenomenon. We show that the Markov chain exhibits a phase transition in
mixing time, as the parameter governing the attraction is varied. Namely, mixing time is $O(n\log n)$ when the temperature is
sufficiently high and $\exp(\Omega(n))$ when temperature is sufficiently low. 
When $\mathcal{G}$ is the complete graph, the model is a projection of the Potts model, whose mixing properties and the
critical temperature have been known previously. However, for any other graph our model is non-reversible and does not
seem to admit a simple Gibbsian description of a stationary distribution. Notably, we demonstrate existence of the 
dynamic phase transition without decomposing the stationary distribution into phases.


\section{Introduction}
In this paper, we introduce the Attracting Random Walks (ARW) model. The motivation of the model is to understand the formation of wealth disparities in an economic network. Consider a network of economic agents, each with a certain number of coins representing their wealth. At each time step, one coin is selected uniformly at random, and moves to a neighbor of its owner with a probability that depends on how wealthy the neighbors are. Those who are well-connected and initially wealthy will tend to accumulate more wealth. We refer to particles instead of coins in what follows. 


This is a flexible model based on a few principles: There are a fixed number of particles moving around on a graph. Movements are asynchronous, and particles make choices about where to move based on their local environment. The model can encompass a variety of situations.  
Further, the model can be extended by allowing for multiple particle types, with intra-- and inter--group attraction parameters, though we do not consider this extension in this paper. There are many more applications beyond the economic application. 
As an interacting particle system, it could be relevant for physics or chemistry applications.

This paper analyzes the Attracting Random Walks model and establishes phase transition properties. The difficulty in
bounding mixing times, particularly in finding lower bounds, is due to the fact that the stationary distribution cannot
be simply formulated. Additionally, the model is not reversible unless the graph is complete (Theorem
\ref{thm:reversibility}), meaning that familiar techniques do not apply. 

We establish the existence of phase transition in mixing time as the attraction parameter, $\beta$, is varied. Slow
mixing for $\beta$ large enough is established by relating the mixing time to a suitable hitting time. Fast mixing for
$\beta$ small enough is proven by a path coupling approach that relates the Attracting Random Walks chain to the simple
(non-interacting) random walk on the same graph (i.e. with $\beta = 0$). 
As a corollary of our main results, we establish properties of the Cheeger cut for the stationary distribution. We find it interesting that even though the stationary distribution is
not known analytically for general graphs, we have shown that it undergoes a phase transition (i.e. develops an exponentially small Cheeger cut) by arguing indirectly via mixing times.

The rest of the paper is structured as follows. We describe the dynamics of the model in Section \ref{sec:model}, along with some possible applications. The remainder of the paper is focused on properties of the Markov chain governing the dynamics. In Section \ref{sec:potts} we discuss a link to the Potts model. Section \ref{sec:phase-transition} proves the existence of phase transition in mixing time for general graphs, and is the main theoretical contribution of this work. In Section \ref{sec:repelling}, we collect partial results on the version of the model in which particles repel each other instead of attracting, a model we call ``Repelling Random Walks.''

\section{The Model}\label{sec:model}
\subsection{Definitions and Main Results}
The model is a discrete time process on a simple graph $\mathcal{G} = (\mathcal{V}, \mathcal{E})$, where $\mathcal{V}$ is the set of vertices and $\mathcal{E}$ is the set of undirected edges. We assume throughout that $\mathcal{G}$ is connected. We write $i \sim j$ if $(i,j) \in \mathcal{E}$. Let $k = |\mathcal{V}|$. Initially, $n$ indistinguishable particles are placed on the vertices of $\mathcal{G}$ in some configuration. Let $x(i)$ be the number of particles at vertex $i$. The particle configuration is updated in two stages, according to a fixed parameter $\beta$:
\begin{enumerate}
\item Choose a particle uniformly at random. Let $i$ be the location of that particle.
\item Move the particle to a vertex $j \sim i, j \neq i$, with probability ${1\over Z} \exp\left(\frac{\beta}{n}
x(j)\right)$. Keep the particle at vertex $i$ with probability ${1\over Z}\exp\left(\frac{\beta}{n} \left(x(i) -1
\right)\right)$, where $Z$ is the normalization constant. 
\end{enumerate}
Let $P$ be the transition probability matrix of the resulting Markov chain. Let $e_i$ denote the $i$th standard basis vector in $\mathbb{R}^k$. Then for two configurations $x$ and $y$ such that $y = x - e_i + e_j$ for $i \sim j$ or $i = j$, we have
$$P(x,y) = \begin{cases}
\frac{x(i)}{n} \frac{\f{x(j)}}{Z} & \text{ if } i \sim j\\
\frac{x(i)}{n} \frac{\f{\left(x(i) - 1 \right)}}Z & \text{ if } i = j
\end{cases}\,,
$$
with $Z=\sum_{l \sim i} \f{ x(l)} + \f{\left(x(i)-1 \right)}$.

The probabilities are a function of the numbers of particles at each vertex, excluding the particle that is to move. This modeling choice means that the moving particle is neutral toward itself, and relates the ARW model to the Potts model, as will be explained below.

When $\beta$ is positive (ferromagnetic dynamics), the particle is more likely to travel to a vertex that has more
particles. Greater $\beta$ encourages stronger aggregation of the particles. On the other hand, taking $\beta < 0$
(antiferromagnetic dynamics) encourages particles to spread. Note that $\beta = 0$ corresponds to the case of
independent (lazy) random walks. 

For an application with $\beta < 0$, consider an ensemble of identical gas particles in a container. We can discretize
the container into blocks. Each block becomes a vertex in our graph. Vertices are connected by an edge whenever the
corresponding blocks share a face. Note that depending on the type of gas, particles may primarily repel each other, in
which case $\beta < 0$, which discourages particles from occupying the same block, does become reasonable. The focus of this paper is the case of $\beta > 0$, though. we collect some results on the $\beta < 0$ case as well.

To get an idea of the effect of $\beta$, Figure \ref{fig:simulation} displays some instances of the Attracting Random Walks model run for $10^5$ steps for different values of $\beta$. The graph is the $8 \times 8$ grid graph, with $n=320$, for an average of $5$ particles per vertex.

\begin{figure}[h!]
\centering
\includegraphics[width=0.75\textwidth]{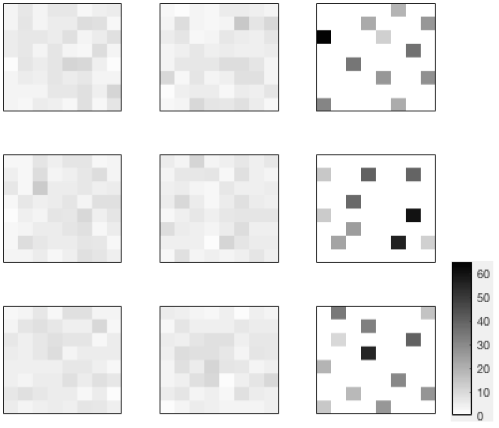}
\caption{Simulation of the Attracting Random Walks model on a grid graph. From left to right: $\beta = 0$, $\beta = 300$, and $\beta = 500$. Three trials are shown for each.}
\label{fig:simulation}
\end{figure}

We now state our main results regarding the phase transition in mixing time. We let $\left \Vert P - Q \right \Vert_{\text{TV}}$ denote the total variation distance between two discrete probability measures $P$ and $Q$, and let $d(X, t) \triangleq \max_{x \in \mathcal{X}} \left \Vert P^t(x,\cdot) - \pi \right \Vert_{\text{TV}}$ be the worst-case (with respect to the initial state) total variation distance for a chain $\{X_t\}$ with stationary distribution $\pi$. Let $t_{\text{mix}}(X, \epsilon) \triangleq \min \left\{t : d(X, t) \leq \epsilon \right \}$ denote the \emph{mixing time} of a chain $\{X_t\}$.

\begin{theorem}\label{thm:slow-mixing}
For any graph $\mathcal{G}$, there exists $\beta_+ > 0$ such that if $\beta > \beta_+$, the mixing time of the ARW model is $\exp(\Omega(n))$.
\end{theorem}

\begin{theorem}\label{thm:fast-mixing}
For any graph $\mathcal{G}$, there exists $\beta_- > 0$ such that if $0 \leq \beta < \beta_-$, the mixing time of the ARW model is $O(n \log n)$.
\end{theorem}
Note that we do not prove that one value $\beta_+ = \beta_-$ satisfies both statements.

Through our analysis of mixing time, we establish a transition in the dynamics of the chain. By standard results, this also indirectly implies that the stationary distribution develops multiple almost disjoint phases for $\beta > \beta_+$, while this is not the case for $\beta < \beta_-$. More precisely, we have the following corollary.
\begin{definition}[Cheeger constant \cite{Peres2017}]
Let $P$ be the transition matrix of a Markov chain that is irreducible and aperiodic. Let $\mathcal{X}$ denote the state space of the chain, and let $\pi$ be the stationary distribution. Define the \emph{edge measure} $Q$ by
\[Q(x,y) \triangleq \pi(x) P(x,y). \]
For two sets $A,B \subset \mathcal{X}$, let $Q(A,B) = \sum_{x \in A, y \in B} Q(x,y)$. For $S \subset \mathcal{X}$, let
\[ \Phi(S) \triangleq \frac{Q(S, S^c)}{\pi(S)}.\]
Finally, the \emph{Cheeger constant} is defined as
\[\Phi_* \triangleq \min_{S : \pi(S) \leq \frac{1}{2}} \Phi(S).\]
\end{definition}
Our results on fast and slow mixing allow us to indirectly bound the Cheeger constant of the Attracting Random Walks chain on a given graph. We obtain the following corollary of Theorems \ref{thm:slow-mixing} and \ref{thm:fast-mixing}.
\begin{corollary}\label{corollary:cheeger}
Fix a graph $\mathcal{G}$, and let $P$ be the transition probability matrix of the Attracting Random Walks chain on $G$. Let $\Phi_*$ be the Cheeger constant of $P$. Then if $0 \leq \beta < \beta_-$ we have $\Phi_* = \frac{1}{O(n \log n)}$. If
$\beta > \beta_+$ then $\Phi_* = \exp(-\Omega(n))$.
\end{corollary}

\subsection{Connection to the Potts Model}\label{sec:potts}
In the case where $\mathcal{G}$ is the complete graph, the Attracting Random Walks model is a projection of Glauber dynamics of the Curie--Weiss Potts model. The Potts model is a multicolor generalization of the Ising model, and the Curie--Weiss version considers a complete graph. In the Curie--Weiss Potts model, the vertices of a complete graph are assigned a color from $[q] = \left\{1, \dots, q \right\}$. Setting $q=2$ corresponds to the Ising model. 

Let $s(i)$ be the color of vertex $i$ for each $1 \leq i \leq n$. Define
\begin{align*}
\delta\left(s(i), s(j) \right) \triangleq  \begin{cases}
    1, & \text{for } s(i) = s(j) \\
    0, & \text{for } s(i) \neq s(j) 
    \end{cases}.
\end{align*}
The stationary distribution of the Potts model, with no external field, is
\begin{align*}
\pi(s) = \frac{1}{Z}\exp\left(\frac{\beta}{n} \sum_{(i,j), i \neq j} \delta\left(s(i), s(j) \right)\right).
\end{align*}

The Glauber dynamics for the Curie--Weiss Potts model are as follows:
\begin{enumerate}
\item Choose a vertex $i$ uniformly at random.
\item Update the color of vertex $i$ to color $k \in [q]$ with probability proportional to $\exp\left(\frac{\beta}{n} \sum_{j \neq i} \delta\left(k, s(j) \right)\right)$.
\end{enumerate}
Observe that the summation $\sum_{j \neq i} \delta\left(k, s(j) \right)$ is equal to the number of vertices, apart from vertex $i$, that have color $k$. Therefore if each vertex in the Potts model corresponds to a particle in the ARW model, and each color in the Potts model corresponds to a vertex in the ARW model, then the ARW model is a projection of the Glauber dynamics for the Potts model. The correspondence is illustrated in Figure \ref{fig:potts-correspondence}. Under the correspondence, the ARW chain is exactly the ``vector of proportions'' chain in the Potts model. 

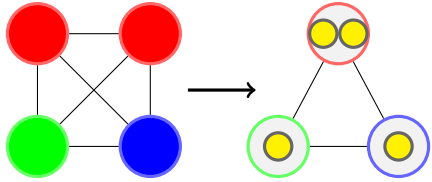
\begin{figure}[h]
\begin{center}
\begin{tikzpicture}[scale=0.5,
pottsred/.style={circle, draw=red!60, fill=red!, very thick, minimum size=8mm},
pottsblue/.style={circle, draw=blue!60, fill=blue!, very thick, minimum size=8mm},
pottsgreen/.style={circle, draw=green!60, fill=green!, very thick, minimum size=8mm},
IRWred/.style={circle, draw=red!60, fill=black!5, very thick, minimum size=8mm},
IRWblue/.style={circle, draw=blue!60, fill=black!5, very thick, minimum size=8mm},
IRWgreen/.style={circle, draw=green!60, fill=black!5, very thick, minimum size=8mm},
coin/.style={circle, draw=black!60, fill=yellow!, very thick, minimum size=1mm},
]

\node[pottsred] at (-6,0)(P1){};
\node[pottsred] at (-3,0)(P2){};
\node[pottsgreen] at (-6,-3)(P4){};
\node[pottsblue] at (-3,-3)(P3){};

\draw[-] (P1.east) -- (P2.west);
\draw[-] (P2.south) -- (P3.north);
\draw[-] (P3.west) -- (P4.east);
\draw[-] (P4.north) -- (P1.south);
\draw[-] (P1.south east) -- (P3.north west);
\draw[-] (P2.south west) -- (P4.north east);

\node[IRWred] at (2,0)(I1){};
\node[IRWgreen] at (0.4,-3)(I2){};
\node[IRWblue] at (3.6,-3)(I3){};

\draw[-] (I1) -- (I2);
\draw[-] (I1) -- (I3);
\draw[-] (I2) -- (I3);

\node[coin] at (1.6,0){};
\node[coin] at (2.4,0){};
\node[coin] at (0.4,-3){};
\node[coin] at (3.6,-3){};

\draw[->, very thick] (-2,-1.5) -- (-0.2,-1.5);

\end{tikzpicture}
\end{center}
\caption{Correspondence of the Curie--Weiss Potts model to the Attracting Random Walks model. A Potts configuration is drawn on the left, and the corresponding ARW configuration is drawn on the right.}
\label{fig:potts-correspondence}
\end{figure}

Let $v(i)$ be the vertex location of the $i$th particle in the ARW model, for $1 \leq i \leq n$. By the correspondence, we show that the stationary distribution of the ARW model is 
\begin{align*}
\pi(x) &= \frac{1}{Z} \binom{n}{x(1), x(2), \dots, x(k)} \exp\left({\frac{\beta}{n} \sum_{(i,j), i \neq j} \delta \left(v(i), v(j) \right)}\right)\\
&= \frac{1}{Z} \binom{n}{x(1), x(2), \dots, x(k)} \exp\left({\frac{\beta}{2n} \sum_{i = 1}^n \left(x\left(v(i)\right) - 1\right)}\right)\\
&= \frac{1}{Z} \binom{n}{x(1), x(2), \dots, x(k)} \exp\left({\frac{\beta}{2n} \sum_{i = 1}^k x(i)^2 - \frac{\beta}{2}}\right)\\
&= \frac{1}{Z'} \binom{n}{x(1), x(2), \dots, x(k)} \exp\left({\frac{\beta}{2n} \sum_{i = 1}^k x(i)^2 }\right).
\end{align*}
Observe that the $\exp\left({\frac{\beta}{2n} \sum_i x(i)^2}\right)$ factor encourages particle aggregation, while the multinomial encourages particle spread. 

The reader is encouraged to refer to \cite{Cuff2012} for a detailed study of the mixing time of the Curie--Weiss Potts model, for different values of $\beta$. For instance, \cite{Cuff2012} shows that there exists $\beta_s(q)$ such that if $\beta < \beta_s(q)$, the mixing time is $\Theta(n \log n)$, and if $\beta > \beta_s(q)$, the mixing time is exponential in $n$. In the $ARW$ context, these results hold with $q$ replaced by $k$. On the other hand, when $\mathcal{G}$ is not the complete graph, the correspondence to the Potts model is lost. In fact, the following can be shown: 
\begin{theorem}\label{thm:reversibility}
For $n \geq 3$, the ARW Markov chain is reversible for all $\beta$ if and only if the graph $\mathcal{G}$ is complete.
\end{theorem}
The non-reversibility can be shown by applying Kolmogorov's cycle criterion, demonstrating a cycle of states (configurations) that violates the criterion.
\begin{lemma}[Kolmogorov's criterion]
A finite state space Markov chain associated with the transition probability matrix $P$ is reversible if and only if for all cyclic sequences of states $i_1, i_2, \dots, i_{l-1}, i_l, i_1$ it holds that
$$\left( \prod_{j=1}^{l-1} P(i_j, i_{j+1}) \right)  P(i_l, i_1) =P(i_1, i_l) \left(\prod_{j=0}^{l-2} P(i_{l-j}, i_{l-j-1}) \right).$$
In other words, the forward product of transition probabilities must equal the reverse product, for all cycles of states. 
\end{lemma}

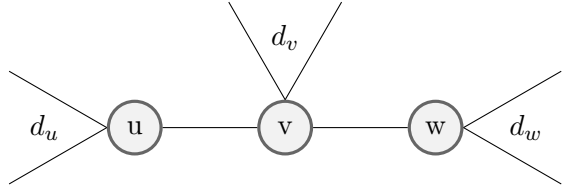
\begin{figure}[h]
\center{
\begin{tikzpicture}[
roundnode/.style={circle, draw=black!60, fill=black!5, very thick, minimum size=7mm},
]
\node[roundnode] at (0,0) (u){u};
\node[roundnode] at (2,0) (v){v};
\node[roundnode] at (4,0) (w){w};
\draw[-] (u.east) -- (v.west);
\draw[-] (v.east) -- (w.west);
\draw (u.west) -- ++(150:1.5cm);
\draw (u.west) -- ++(210:1.5cm);
\draw (v.north) -- ++(60:1.5cm);
\draw (v.north) -- ++(120:1.5cm);
\draw (w.east) -- ++(30:1.5cm);
\draw (w.east) -- ++(330:1.5cm);
\node at (-1.2,0){$d_u$};
\node at (2,1.2){$d_v$};
\node at (5.2,0){$d_w$};
\end{tikzpicture}
\vspace{12pt}
}
\caption{Initial state of a cycle that breaks Kolmogorov's criterion.} \label{fig:kolmogorov}
\end{figure}

\begin{proof}[Proof of Theorem \ref{thm:reversibility}]
First, if the graph is complete, then the chain is a projection of Glauber dynamics, which is automatically reversible. Now suppose $\mathcal{G}$ is not complete. We apply Kolmogorov's cycle criterion. In the ARW model, a state is a particle configuration. A cycle of states is then a sequence of particle configurations such that 
\begin{enumerate}
\item Subsequent configurations differ by the movement of a single particle.
\item The first and last configurations are the same.
\end{enumerate}

If $\mathcal{G}$ is not a complete graph, then it is straightforward to show that there exist three vertices $u \sim v \sim w$ such that $u \nsim w$. Now we demonstrate a cycle of states that breaks Kolmogorov's criterion. We have the following situation, illustrated by Figure \ref{fig:kolmogorov}. 
The values $d_u$, $d_v$, and $d_w$ indicate the degrees of the vertices, excluding the named vertices. Place $n-2$ particles at $u$ and $2$ particles at $v$. The particle movements are as follows: $v \rightarrow u$, $v \rightarrow w$, $u \rightarrow v$, $w \rightarrow v$. 

For clarity, let $f(z) = \exp\left(\frac{\beta}{n} z\right)$. The forward transition probabilities are
\scriptsize
\begin{align*}
\left(\frac{2}{n}\frac{f(n-2)}{f(n-2)+f(1)+1+d_v}\right)
\left(\frac{1}{n} \frac{1}{f(n-1) + 1+1+d_v}\right)
\left(\frac{n-1}{n} \frac{1}{f(n-2) + 1+ d_u}\right)
\left(\frac{1}{n} \frac{f(1)}{f(1)+1+d_w}\right).
\end{align*}
\normalsize
The reverse transition probabilities are
\scriptsize
\begin{align*}
\left(\frac{2}{n}\frac{1}{f(n-2)+f(1)+1+d_v}\right)
\left(\frac{1}{n}\frac{f(n-2)}{f(n-2)+1+f(1)+d_v}\right)
\left(\frac{1}{n}\frac{1}{1+1+d_w}\right)
\left(\frac{n-1}{n}\frac{f(1)}{f(n-2) + f(1)+d_u}\right).
\end{align*}
\normalsize
Canceling factors that appear in both products, we are left comparing 
$$\left(f(n-1) + 1+1+d_v\right)\left(f(n-2) + 1+ d_u\right) \left(f(1)+1+d_w\right)$$ to
$$\left(f(n-2)+1+f(1)+d_v\right)\left(1+1+d_w \right)\left(f(n-2) + f(1)+d_u\right).$$
Observe that $f(z_1) f(z_2) = f(z_1+z_2)$. Taking leading terms, the first product is therefore a degree-$(2n-2)$ polynomial in $e^{\beta}$. Since $n-2 \geq 1$, the second is a degree-$(2n-4)$ polynomial in $e^{\beta}$. These polynomials have a finite number of solutions for $e^{\beta}$, and therefore $\beta$ itself. Therefore the Markov chain is not reversible. 
\end{proof}

\section{Mixing Time on General Graphs}\label{sec:phase-transition}
In this section, we show the existence of phase transition in mixing time in the ARW model when $\beta$ is varied, for a
general fixed graph. First, we show exponentially slow mixing for $\beta$ suitably large, namely prove Theorem
\ref{thm:slow-mixing} by relating mixing times to hitting times. Next, we show polynomial time mixing for small values
of $\beta$. The proof is by an adaptation of path coupling. We use definitions and notations on Markov chains from~\cite{Peres2017}.


\subsection{Slow Mixing}
The idea of the proof of slow mixing is to show that with substantial probability, the chain takes an exponential time to access a constant portion of the state space. We now outline the proof, deferring the proofs of the lemmas. First we state a helper lemma.
\begin{lemma}\label{helper-lemma}
For any graph $\mathcal{G} = \left( \mathcal{V}, \mathcal{E} \right)$, there exists a vertex $v \in \mathcal{V}$ such that for the set of configurations $S_v \triangleq \{x : x(v) = \max_w x(w) \}$, it holds that $\pi(S_v) \geq \nicefrac{1}{k}$. In other words, the states where $v$ has the greatest number of particles contribute at least $\nicefrac{1}{k}$ to the stationary probability mass.
\end{lemma}
By Lemma \ref{helper-lemma}, there exists a vertex $v$ such that $\pi(S_v) \geq \nicefrac{1}{k}$. Choose any other vertex $u$. Whenever $x(u) > \nicefrac{n}{2}$, we can be sure that $v$ is not the maximizing vertex, and therefore that a set of states having at least $\nicefrac{1}{k}$ mass under the stationary measure has not been reached. It therefore suffices to lower bound the time until vertex $u$ has lost sufficient particles for vertex $v$ to have the maximum number of particles.

Let $T_x \triangleq  \inf \{t : X_t(u) \leq \frac{1}{2}n, X_0 = x\}.$ If the probability that $\{X_t\}$ has reached the set $\{x  \in \Omega : x(0) \leq \nicefrac{n}{2}\}$ by time $t$ is less than some $p$, then the total variation distance at time $t$ is at least $(1-p)\frac{1}{k}$. Therefore we get the following relationship between the mixing time and hitting time:
\begin{proposition}\label{proposition:mix-hit}
$$t_{\text{mix}}\left(X, (1-p)\frac{1}{k} \right) \geq \inf \left \{ t : \min_x \mathbb{P}\left(T_x \leq t \right) \geq p \right \}.$$
\end{proposition}

The problem now reduces to lower bounding this hitting time. The idea is that when particles leave vertex $u$, there is a strong drift back to $u$. However, controlling the hitting times of a multidimensional Markov chain is challenging, and direct comparison is difficult to establish. We instead reason by comparison to another Markov chain, $Z$, which lower-bounds the particle occupancy at vertex $u$.

Let $l(w)$ be the length of the shortest path connecting vertex $u$ to  vertex $w$. Let $\tilde{X}_t$ be a projection of the $X_t$ chain defined by $\tilde{X}_t(d) \triangleq \sum_{w : l(w) = d} X_t(w)$, and let $\tilde{\Omega}$ be its state space. In other words, the $d$th coordinate of the projected chain counts the number of particles that are a distance $d$ away from vertex $u$. Note that $\tilde{X}_t(0) =  X_t(u)$. We let $F$ denote this projection, writing, $\tilde{X} = F(X)$. For any $0 < \delta < \nicefrac{1}{2}$, define 
$$T_x(\delta) \triangleq \inf \{t : X_t(u) \leq (1-\delta)n, X_0 = x\} = \inf \{t: \tilde{X}_t(0) \leq (1-\delta)n, \tilde{X}_0 = F(x) \}.$$
For some $\delta > 0$ to be determined, let 
\[S \triangleq \{x \in \tilde{\Omega}: x(0) > (1-\delta)n\} \text{~~and~~} S^c \triangleq \tilde{\Omega} \setminus S.\] 
We now build a chain $Z$ on $\tilde{\Omega}$ coupled to $\tilde{X}$ such that as long as $\tilde{X}_t \in S$, $Z_t(0) \stleq \tilde{X}_t(0)$. Then $T_x(\delta) \stgeq \inf_t \{Z_t \in S^c \}$. The remainder the proof of slow mixing is as follows.
\begin{enumerate}
\item Construct a lower-bounding comparison chain $Z$ satisfying $Z_t(0) \stleq \tilde{X}_t(0)$ when $t \leq T_x(\delta)$.
\item Compute $\mathbb{E}_{\pi_Z}\left[Z(0) \right]$ and use a concentration bound to show that $Z(0) \sim \pi_Z(0)$ places exponentially little mass on the set $S^c$.
\item Comparing the chain $X$ to $Z$, show that $X$ takes exponential time to achieve $X(u) \leq (1- \delta)n$. The result is complete by $1 -\delta > \nicefrac{1}{2}$.
\end{enumerate}

We now define the lower-bounding comparison chain $Z$, which is a chain on $n$ independent particles. These particles move on the discrete line with points $\{0, 1, \dots, D\}$, where $D = diam(\mathcal{G})$. We first describe the case $D \geq 2$. Since the comparison needs to hold only when $\tilde{X}_t(0) \geq (1-\delta)n$, we assume that $\tilde{X}_t(0) \geq (1-\delta)n$. The idea is to identify a uniform constant lower bound on the probability of a particle moving closer to $u$ under this assumption, which tells us that once the particle is at $u$, there is a high probability of remaining there. 

Let $\mathcal{N}(u)$ denote the neighbourhood of $u$, i.e. $\mathcal{N}(u) = \{w: w \sim u\}$. In the $X$ chain, when a particle is at a vertex $w \notin \{u\} \cup \mathcal{N}(u)$, its probability of moving to any one of its neighbors is at least $$p \triangleq \frac{1}{e^{\beta \delta} + \Delta},$$ where $\Delta$ is the maximum degree of the graph. This is because the lowest probability when $\beta$ is large corresponds to placing all $\delta n$ movable particles at some other neighbor of $w$. When a particle is at a vertex $u$, it stays there with probability at least 
$$q \triangleq \frac{\exp\left(\beta (1-\delta) - \frac{\beta}{n}\right)}{\exp \left(\beta (1-\delta) - \frac{\beta}{n}\right)+e^{\beta \delta} + \Delta - 1},$$ 
When a particle is at a vertex $w \in \mathcal{N}(u)$, it moves to $u$ with probability at least $$\frac{\exp\left(\beta (1-\delta)\right)}{\exp\left(\beta (1-\delta)\right)+e^{\beta \delta} + \Delta - 1} > q .$$ 
Note that $q > p$.

The transitions of the $Z$ chain are chosen in order to maintain comparison. At each time step, a particle is selected uniformly at random. When the chosen particle is located at $d \notin \{0,1\}$, the particle moves to $d-1$ with probability $p$ and moves to $\min\{d+1, D\}$ with probability $(1-p)$. When the chosen particle is located at $d \in \{0,1\}$, it moves to $0$ with probability $q$, and moves to $d+1$ with probability $1-q$. The transition probabilities for single particle movements are depicted in Figure \ref{fig:Z-chain}. When $D = 1$ (i.e., $\mathcal{G}$ is the complete graph), we instead have the transitions depicted by Figure \ref{fig:Z-chain-small}. Lemma \ref{lemma:comparison} establishes the comparison. 

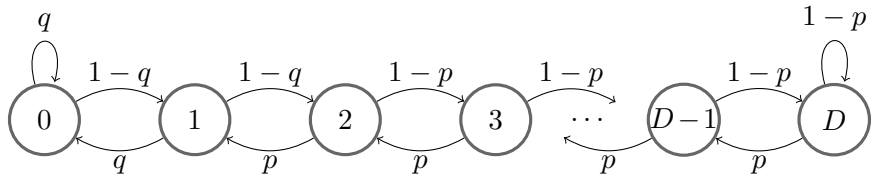
\begin{figure}[h]
\centering
\begin{tikzpicture}[
roundnode/.style={circle, draw=black!60, very thick, inner sep=0pt,
  text width=9mm, align=center},
]
\node[roundnode] at (0,0) (0){$0$};
\node[roundnode] at (2,0) (1){$1$};
\node[roundnode] at (4,0) (2){$2$};
\node[roundnode] at (6,0) (3){$3$};
\node[roundnode] at (8.5,0) (D-1){$D-1$};
\node[roundnode] at (10.5,0) (D){$D$};
\node[roundnode,draw=none] at (6.5,0)(n1){};
\node[roundnode,draw=none] at (8,0)(n2){};
\node at (7.25,0){$\dots$};

\path (0) edge [loop above] node {$q$} (0);
\path (D) edge [loop above] node {$1-p$} (D);
\draw[->] (0) to [out=30, in=150] (1);
\node at (1,0.6){$1-q$};
\draw[->] (1) to [out=210, in=330] (0);
\node at (1,-0.6){$q$};
\draw[->] (1) to [out=30, in=150] (2);
\node at (3,0.6){$1-q$};
\draw[->] (2) to [out=210, in=330] (1);
\node at (3,-0.6){$p$};
\draw[->] (2) to [out=30, in=150] (3);
\node at (5,0.6){$1-p$};
\draw[->] (3) to [out=210, in=330] (2);
\node at (5,-0.6){$p$};
\draw[->] (3) to [out=30, in=150] (n2);
\node at (7,0.6){$1-p$};
\draw[->] (D-1) to [out=210, in=330] (n1);
\node at (7.5,-0.6){$p$};
\draw[->] (D-1) to [out=30, in=150] (D);
\node at (9.5,0.6){$1-p$};
\draw[->] (D) to [out=210, in=330] (D-1);
\node at (9.5,-0.6){$p$};
\end{tikzpicture}
\caption{Single-particle Markov chain from the $Z$ chain ($D \geq 2$)}
\label{fig:Z-chain}
\end{figure}

\begin{figure}[h]
\centering
\begin{tikzpicture}[
roundnode/.style={circle, draw=black!60, very thick, inner sep=0pt,
  text width=9mm, align=center},
]
\node[roundnode] at (0,0) (0){$0$};
\node[roundnode] at (2,0) (1){$1$};
\path (0) edge [loop above] node {$q$} (0);
\path (1) edge [loop above] node {$1-q$} (1);
\draw[->] (0) to [out=30, in=150] (1);
\node at (1,0.6){$1-q$};
\draw[->] (1) to [out=210, in=330] (0);
\node at (1,-0.6){$q$};
\end{tikzpicture}
\caption{Single-particle Markov chain from the $Z$ chain ($D=1$)}
\label{fig:Z-chain-small}
\end{figure}
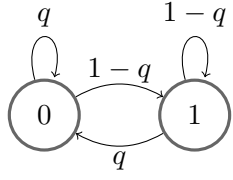

Let $\pi_Z$ denote the stationary distribution of the $Z$ chain, and let $\lambda(w)$ be the probability according to $\pi_Z$ of a particular particle being located at vertex $w$ in the line graph. The following results about the $Z$ chain are required to complete the proof.
\begin{lemma}\label{lemma:comparison}
For a configuration $x \in \Omega$, set $Z_0 = \tilde{X}_0 = F(x)$. As long as $t \leq T_x(\delta)$, the chain $Z_t$ satisfies 
$$\sum_{r = 0}^d  Z_t(r) \stleq \sum_{r = 0}^d \tilde{X}_t(r)$$ 
for all $d \in \{0, 1, \dots, D\}$ and $t \in \{0, 1, 2, \dots\}$. In particular, $Z_t(0) \stleq \tilde{X}_t(0) $.
\end{lemma}

\begin{lemma}\label{lemma:concentration}
Recall that $D = diam(\mathcal{G})$. Let $\delta = \frac{1}{3D}$ and fix $0 < \overline{\epsilon} < \nicefrac{\delta}{2}$. For all $\beta$ large enough, $\mathbb{E}_{\pi(Z)} \left[Z(0)\right] \geq (1-\delta + \overline{\epsilon}) n$. Moreover, 
$$\mathbb{P}_{\pi(Z)} \left(Z(0) \leq \mathbb{E}_{\pi(Z)}\left[Z(0)\right] - \overline{\epsilon}n \right) \leq 2 \exp \left(-2 \overline{\epsilon}^2 n\right),$$ which implies
$$\mathbb{P}_{\pi(Z)} \left(Z(0) \leq (1-\delta)n \right) \leq 2 \exp \left(-2 \overline{\epsilon}^2 n\right).$$
\end{lemma}

\begin{proof}[Proof of Theorem \ref{thm:slow-mixing}]
Recall the choices of $u$ and $v$ above. 
Lemma \ref{lemma:concentration} tells us that the $Z$ chain places exponentially little stationary mass on the set $S^c$. We now combine this fact with the comparison established in Lemma \ref{lemma:comparison}.

Recall $T_x(\delta) = \inf \{t : X_t(u) \leq (1-\delta)n, X_0 = x\} = \inf \{t: \tilde{X}_t(0) \leq (1-\delta)n, \tilde{X}_0 = F(x) \}.$ Applying Proposition \ref{proposition:mix-hit} with $p = \nicefrac{1}{2}$,
\begin{align*}
t_{\text{mix}}\left(X, \frac{1}{2k} \right) 
&\geq \min \left \{ t: \min_x \mathbb{P}\left(T_x\left(\frac{1}{2}\right) \leq t \right) \geq \frac{1}{2}\right \}.
\end{align*}
Since $\nicefrac{1}{2} < 1- \delta$, it also holds that 
\begin{align}
t_{\text{mix}}\left(X, \frac{1}{2k} \right) 
&\geq \min \left \{ t: \min_x \mathbb{P}\left(T_x\left(\delta \right) \leq t \right) \geq \frac{1}{2}\right \} \nonumber\\
&= \min \left \{ t: \mathbb{P}(T_x (\delta)\leq t) \geq \frac{1}{2}, \forall x \in \Omega \right \} \nonumber\\
&= \min \left \{ t: \mathbb{P}(T_x(\delta) \leq t) \geq \frac{1}{2}, \forall x \in S \right \}. \label{mix-bound}
\end{align}
The last equality is due to the fact that $\mathbb{P}(T_x(\delta) \leq t) = 1$ for all $x$ in $S^c$.

Additionally define 
$$T_x^Z \triangleq \inf  \left \{t :Z_t \in \mathcal{S}^c , Z_0 = F(x) \right\}.$$
Now because $Z_t$ is a lower-bounding chain, it holds that $$\mathbb{P}\left(T_x(\delta) \leq t \right) \leq \mathbb{P}\left(T_x^Z \leq t \right)$$ for all $x \in S$ and $t \geq 0$. Therefore, 
$$t_{\text{mix}}\left(X, \frac{1}{2k} \right)  \geq \min \left \{ t: \mathbb{P}\left(T_x^Z \leq t\right) \geq \frac{1}{2}, \forall x \in S \right \} . $$
Finally, from Lemma \ref{lemma:concentration} we know that $\pi_Z(S^c) \leq 2\exp \left(-2 \overline{\epsilon}^2n\right)$. Suppose that $Z_0$ is distributed according to $\pi_Z$ and consider the hitting time $T_{\pi_z}^Z$. It holds that \[T_{\pi_z}^Z \stgeq \text{Geom}\left(2\exp \left(-2 \overline{\epsilon}^2n\right)\right).\] Therefore, $t = e^{\Theta(n)}$ time is required for $\mathbb{P}\left(T_{\pi_Z}^Z \leq t\right) \geq \frac{1}{2}$. The same is true when $Z_0 = x$, for some $x \in S$. 
Therefore $\min \left \{ t: \mathbb{P}\left(T_x^Z \leq t\right) \geq \frac{1}{2}, \forall x \in S \right \} = e^{\Theta(n)}$ and $t_{\text{mix}}\left(X, \frac{1}{2k} \right) = e^{\Omega(n)}$, which proves Theorem \ref{thm:slow-mixing}. 
\end{proof}



We now provide the deferred proofs.
\begin{proof}[Proof of Lemma \ref{helper-lemma}]
By the Union Bound,
\[\sum_{v \in \mathcal{V}} \pi(S_v) \geq \mathbb{P}_{\pi} \left(\cup_{v \in \mathcal{V}} \left \{x(v) = \max_w x(w)\right \} \right) = 1. \qedhere \]
\end{proof}

\begin{proof}[Proof of Lemma \ref{lemma:comparison}]
We show that there exists a coupling $(\tilde{X}_t, Z_t)$ satisfying 
\begin{align}
\sum_{r = 0}^d  Z_t(r) \leq \sum_{r = 0}^d \tilde{X}_t(r) \label{eq:star}
\end{align}
for all $d \in \{0, 1, \dots, D\}$ and $t \leq T_x(\delta)$. Since $Z_0 = \tilde{X}_0$, we can pair up the particles at time $t = 0$ and design a synchronous coupling, i.e. when a certain particle is chosen in the $\tilde{X}$ process, its copy is chosen in the $Z$ chain. We design the coupling so that for each particle, the $\tilde{X}$--copy is at least as close to $0$ as the $Z$--copy, for all $t \leq T_x(\delta)$. Note that this implies $\eqref{eq:star}$ for all $d \in \{0,1, \dots, D\}$ and $t < T_x(\delta)$.
The uniformity of $p$ and $q$ over all configurations in $S$ ensures that the coupling will maintain the requirement \eqref{eq:star}, which is established by induction on $t$. The following analysis applies to both $D \geq 2$ and $D= 1$ by considering the relevant cases.

The base case $(t=0)$ holds since $Z_0 = \tilde{X}_0$. Suppose that at time $t < T_x(\delta)$, each particle in the $\tilde{X}$ chain is at least as close to $0$ as its copy in the $Z$ chain. We will show that the same property holds for time $t+1$. First consider a particle located at $0$ in the $Z$ chain. By the inductive hypothesis, its copy must be located at $0$ in the $\tilde{X}$ process also, and the corresponding particle in the $X$ chain must be at $u$. The probability of the particle staying at $0$ in the $Z$ chain is smaller than the probability of the corresponding particle staying at $u$ in the $X$ chain, since $q$ is a uniform lower bound on the probability of staying at $u$. Therefore in this case, the property is maintained in the next time step. 

Next consider a particle located at vertex $d \neq 0$ in the $Z$ chain and suppose its copy is located at vertex $d'$ in the $\tilde{X}$ process. By the inductive hypothesis, $d' \leq d$. If $d' < d -1$, then clearly the property is maintained in the next step. It remains to consider the cases $d' = d$ and $d' = d -1$. Consider the case $d = d'$. We couple the particles so that if the particle in the $Z$ chain moves left to vertex $d-1$, then the particle in the $\tilde{X}$ process makes the same transition. This coupling is possible by the uniformity of $p$ and $q$. Otherwise, the particle in the $Z$ chain moves right, and the property is maintained. 

Next consider the case $d' = d - 1$. It suffices to design a coupling such that if the particle in the $\tilde{X}$ process moves right, then so does the particle in the $Z$ chain. If $d \geq 3$, this is possible due to the fact that $1-p$ is a uniform upper bound on the probability of moving right from these states. Next suppose that $d = 2$ and $d' =1$. The particle in the $\tilde{X}$ process moves right with probability upper-bounded by $1-q$, which is smaller than $1 - p$ for $\delta$ sufficiently small and $n$ sufficiently large. Therefore we can ensure the property in the next step. Finally, suppose $d=1$ and $d' = 0$. Due to the fact that $1-q$ is a uniform upper bound on the probability of moving right from these states, we can again construct a coupling that maintains the property.
\end{proof}

To prove Lemma \ref{lemma:concentration}, we need the stationary probability $\lambda(0)$.
\begin{proposition}\label{prop:stationary-probabilities}
It holds
\begin{align*}
\lambda(0) &= \begin{cases}
q & \text{if } D = 1\\
q\left[1 + \frac{\left(1-q \right)^2}{p} \left(\frac{p}{1-p} \right)^{2-D} \left(\frac{1- \left( \frac{p}{1-p} \right)^{D-1}}{1-\frac{p}{1-p}} \right)\right]^{-1} & \text{if } D \geq 2
\end{cases}.
\end{align*}
\end{proposition}
The proof of Proposition \ref{prop:stationary-probabilities} is deferred to the appendix.

\begin{proof}[Proof of Lemma \ref{lemma:concentration}]
When $D = 1$, we have $\lambda(0) = q$. Since $q(\beta) \to 1$ as $\beta \to \infty$, it holds that $\lambda(0) \geq 1- \delta + \overline{\epsilon}$ for $\beta$ large enough. Next, for $D \geq 2$ we have
\begin{align}
\lambda(0) &= \frac{q}{1 + \frac{\left(1-q \right)^2}{p} \left(\frac{p}{1-p} \right)^{2-D} \left(\frac{1- \left( \frac{p}{1-p} \right)^{D-1}}{1-\frac{p}{1-p}} \right)} \nonumber \\
&= \frac{q}{1 + \frac{\left(1-q \right)^2}{p} \left(\frac{p}{1-p} \right)^{1-D} \left(\frac{\frac{p}{1-p}- \left( \frac{p}{1-p} \right)^{D}}{1-\frac{p}{1-p}} \right)} \nonumber \\
&\geq \frac{q}{1 + \frac{\left(1-q \right)^2}{p} \left(\frac{p}{1-p} \right)^{1-D} \left(\frac{\frac{p}{1-p}}{1-\frac{p}{1-p}} \right)} \nonumber \\
&= \frac{q}{1 + \frac{\left(1-q \right)^2}{1-2p}  \left(\frac{p}{1-p} \right)^{1-D}}.\label{eq:lambda}
\end{align}
We show that $\lambda(0) \geq 1 - \delta + \overline{\epsilon}$ for $\beta$ and $n$ large enough. Again, for $\beta$ large enough, we have $q \geq 1 - \delta + 2 \overline{\epsilon}$. Next,
\begin{align}
1-q &= \frac{e^{\beta \delta} + \Delta - 1}{\exp \left( \beta(1-\delta) - \frac{\beta}{n} \right) + e^{\beta \delta} + \Delta - 1} \nonumber\\
&\leq \frac{e^{\beta \delta} + \Delta - 1}{\exp \left( \beta(1-\delta) - \frac{\beta}{n} \right)} \nonumber \\
&= \exp\left(\beta\left(\frac{2}{3D} -1 + \frac{1}{n}\right) \right) + (\Delta - 1)\exp \left( \beta \left( \frac{1}{3D} -1 + \frac{1}{n} \right)  \right) \label{eq:dominate}\\
&\leq \exp\left(\beta\left(\frac{1}{D} -1\right) \right), \nonumber
\end{align}
where the last inequality holds for $\beta$ and $n$ large enough, since the first term of \eqref{eq:dominate} dominates. Since $p < \frac{1}{\Delta + 1}$, we have $1 - 2p > \frac{\Delta - 1}{\Delta + 1}$. Finally,
\begin{align*}
\left( \frac{p}{1-p}\right)^{1-D} &= \left( e^{\beta \delta} + \Delta - 1 \right)^{D-1} = \left( e^{\nicefrac{\beta}{(3D)}}+ \Delta - 1 \right)^{D-1} \leq \left( e^{\nicefrac{\beta}{(2D)}} \right)^{D-1} < e^{\nicefrac{\beta}{2}},
\end{align*}
where the first inequality holds for $\beta$ large enough. Substituting into \eqref{eq:lambda}, we obtain for $\beta$ and $n$ large enough
\begin{align*}
\lambda(0) &\geq \frac{1 - \delta + 2 \overline{\epsilon}}{1 + \frac{\Delta + 1}{\Delta -1} \exp\left(2 \beta \left(\frac{1}{D}-1\right) + \frac{\beta}{2} \right)} = \frac{1 - \delta + 2 \overline{\epsilon}}{1 + \frac{\Delta + 1}{\Delta -1} \exp\left( \beta \left(\frac{2}{D}-\frac{3}{2}\right)  \right)} \geq 1- \delta + \overline{\epsilon},
\end{align*}
where the second inequality holds for $\beta$ large enough, due to $\nicefrac{2}{D} - \nicefrac{3}{2} < 0$ for $D \geq 2$. 

We conclude that the expectation is linearly separated from the boundary:
\[\mathbb{E}_{\pi_Z} \left[ Z(0) \right] = \lambda(0)n \geq \left(1 - \delta + \overline{\epsilon} \right)n.\] 

Next we show concentration. Label all the particles, and define $U_i = 1$ if particle $i$ is at vertex $0$ in the line graph, and $U_i = 0$ otherwise.
Then $Z(0) = \sum_{i} U_{i}$, and $U_{i}$ is independent of $U_{j}$ for all $i \neq j$. Applying Hoeffding's inequality,
$$\mathbb{P}_{\pi_Z} \left(\left | Z(0) - \mathbb{E}_{\pi(Z)}[Z(0)] \right| \geq cn  \right) \leq 2\exp \left(- \frac{2(cn)^2}{n} \right) = 2\exp \left(- 2c^2n \right).$$
for $c > 0$. Let $c =\overline{\epsilon}$. Then the above implies 
\begin{align*}
&\mathbb{P}_{\pi_Z} \left(Z(0) \leq \mathbb{E}_{\pi_Z}[Z(0)] - \overline{\epsilon} n \right) \leq 2\exp \left(- 2\overline{\epsilon}^2n \right)\\
\implies &\mathbb{P}_{\pi_Z} \left(Z(0) \leq \left(1 - \delta  \right)n \right) \leq 2\exp \left(-2 \overline{\epsilon}^2n\right). \qedhere
\end{align*}
\end{proof}

\subsection{Fast Mixing}
The proof is by a modification of path coupling, which is a method to find an upper bound on mixing time through contraction of the Wasserstein distance. \footnote{An alternative prove of fast mixing is to use a variable-length
path coupling, as introduced in \cite{Hayes2007}. For further details, see \cite{Gaudio2020}.} The following definition can be found in \cite{Peres2017}, pp. 189.

\begin{definition}[Transportation metric]
Given a metric $\rho$ on a state space $\Omega$, the associated transportation metric $\rho_T$ for two probability distributions $\mu$ and $\nu$ is defined as
\begin{align*}
\rho_T(\mu, \nu) \triangleq \inf_{X \sim \mu, Y \sim \nu} \mathbb{E}[\rho(X,Y)]
\end{align*}
where the infimum is over all couplings of $\mu$ and $\nu$ on $\Omega \times \Omega$.
\end{definition}

\begin{definition}[Wasserstein distance]
Let $P$ be the transition probability matrix of a Markov chain on a state space $\Omega$, and let $\rho$ be a metric on $\Omega$. The Wasserstein distance $W_{\rho}^P(x,y)$ of two states $x, y \in \Omega$ with respect to $P$ and $\rho$ is defined as follows:
$$W_{\rho}^P(x,y) \triangleq \rho_T\left(P(x, \cdot), P(y, \cdot) \right) = \inf_{X_1 \sim P(x, \cdot), Y_1 \sim P(y, \cdot)} \mathbb{E}_{X_1, Y_1} \left[ \rho(X_1, Y_1) \right].$$
In other words, the Wasserstein distance is the transportation metric distance between the next state distributions from initial states $x$ and $y$.
\end{definition}

The following lemma is the path coupling result which can be found in \cite{Bubley1975} and \cite{Peres2017}. Given a Markov chain on state space $\Omega$ with transition probability matrix $P$, consider a connected graph $\mathcal{H} = \left(\Omega, \mathcal{E}_{\mathcal{H}}\right)$, i.e. the vertices of $\mathcal{H}$ are the states in $\Omega$ and the edges are $\mathcal{E}_{\mathcal{H}}$. Let $l$ be a ``length function'' for the edges of $\mathcal{H}$, which is an arbitrary function $l :  \mathcal{E}_{\mathcal{H}} \to [1, \infty)$. For $x,y \in \Omega$, define $\rho(x,y)$ to be the path metric, i.e. $\rho(x,y)$ is the length of the shortest path from $x$ to $y$ in terms of $l$ and $\mathcal{H}$.

\begin{lemma}[Path Coupling]\label{lemma:path-coupling}
Under the above construction, if there exists $\delta > 0$ such that for all $x,y$ that are connected by an edge in $\mathcal{H}$ it holds that
$$W_{\rho}^P(x,y) \leq (1-\delta) \rho (x,y),$$
then $$d(X, t) \leq (1 - \delta)^t \text{diam}(\Omega),$$
where $\text{diam}(\Omega) = \max_{x,y \in \Omega} \rho(x,y)$ is the diameter of the graph $\mathcal{H}$ with respect to $\rho$. 
\end{lemma}


Our proof of rapid mixing for small enough $\beta$ relies on rapid mixing of a single random walk. The following lemma demonstrates the existence of a contracting metric for a single random walk. It is possible that such a result appears elsewhere, but we are not aware of a published proof.

\begin{lemma}\label{single-RW} Consider a random walk on $\mathcal{G}$ which makes a uniform choice among staying or moving to any
of the neighbors and denote by $Q$ its transition matrix. 
Let $d(x,y)$ be the expected meeting time of two independent copies of a random walk on a graph started from states $x$
and $y$. Then $d(x,y)$ is a metric and $Q$ contracts the respective Wasserstein distance. In particular, 
$$W_{d}^{Q}(x,y) \leq \left(1 - \frac{1}{d_{\text{max}}} \right) d(x,y),$$ where $d_{\text{max}} = \max_{x,y} d(x,y)$.
Furthermore, if $x \sim y$, then 
\begin{equation}\label{eq:srw_strong}
	W_{d}^{Q}(x,y) \leq \left(1 - \frac{1}{d'_{\text{max}}} \right) d(x,y), 
\end{equation}where $d'_{\text{max}} = \max_{x,y: x \sim y} d(x,y)$. \footnote{The statement for $x \sim y$ was pointed out by the reviewer.}
\end{lemma}
\begin{remark} In fact, we can show a stronger result (i.e. with a smaller value in the place of $d_{\max}$): we can allow
arbitrary Markovian coupling between two copies of the random walk and define 
$d(x,y)$ to be the meeting time under that coupling. 
\end{remark}

In order to apply path coupling, we let $\mathcal{H} = \left(\Omega, \mathcal{E}_{\mathcal{H}} \right)$ be a graph on particle configurations, where $(x,y) \in \mathcal{E}_{\mathcal{H}}$ whenever $y = x - e_i + e_j$ for some pair of distinct vertices $i$ and $j$ in $\mathcal{G}$. In other words, $x$ and $y$ differ by the position of a single particle. Note that $i$ and $j$ need not be neighboring vertices in $\mathcal{G}$. For such a pair of neighboring configurations $(x,y)$, let $l(x,y) = d(i,j)$. Clearly, $l(x,y) \geq \mathbbm{1}\{x \neq y\}$. Now for any two configurations $x, y \in \Omega$, let $\rho(x,y)$ denote the path metric induced by $\mathcal{H}$ and $l(\cdot, \cdot)$. We show that $\rho(x,y) = l(x,y)$ for neighboring configurations.
\begin{proposition}\label{prop:distance}
For any two configurations $x, y$ such that $y = x - e_i + e_j$, it holds that $\rho(x,y) = l(x,y)$.
\end{proposition}
Let $P_x(i, \cdot)$ be the probability distribution of the next location of the selected particle, when it is initially located at vertex $i \in \mathcal{V}$ in configuration $x$. Recall that $Q(i, \cdot)$ is the probability distribution of the next location of a simple random walk on $\mathcal{G}$, initially located at vertex $i$. Note that when $\beta = 0$, it holds that $P_x(i, \cdot) = Q(i, \cdot)$. When $\beta$ is small, $P_x(i, \cdot) \approx Q(i, \cdot)$. Lemma \ref{lemma:close-distributions} quantifies this statement.
\begin{lemma}\label{lemma:close-distributions}
For all configurations $x$ and vertices $i \in \mathcal{G}$, it holds that
$$\left \Vert P_x(i, \cdot) - Q(i, \cdot) \right \Vert_{\text{TV}} \leq \frac{e^{\nicefrac{\beta}{2}} - 1}{e^{\nicefrac{\beta}{2}} + 1}.$$
\end{lemma}

Next, consider two neighbouring configurations $x$ and $y$. Because only the position of one particle is different between the two configurations, $P_x(v, \cdot) \approx P_y(v,\cdot)$. The following lemma makes this precise.
\begin{lemma}\label{lemma:TV-bound-same-vertex}
Let $x$ and $y$ be neighbouring configurations. Recall that $\Delta$ is the maximum degree of the vertices in $\mathcal{V}$. The following holds:
$$\left \Vert P_x(v, \cdot) - P_y(v, \cdot) \right \Vert_{\text{TV}} \leq \frac{\Delta + 1 }{n} \beta.$$
\end{lemma}

With these results stated, we prove Theorem \ref{thm:fast-mixing}.
\begin{proof}[Proof of Theorem \ref{thm:fast-mixing}]
Suppose $d(i,j)\ge 1\{i\neq j\}$ is a metric on $\mathcal{G}$ such that a single-particle random walk's kernel $Q$ satisfies
\begin{align}\label{eq:rq_con}
	W_d^Q(i,j) \le (1-\delta) d(i,j) 
\end{align}
for all $i \neq j$ and $d(i,j) \le d_{max}$. Note that the existence of such a metric $d(\cdot, \cdot)$ was established in Lemma~\ref{single-RW} with an estimate of $\delta = \nicefrac{1}{d_{max}}$.

Now we wish to bound $W_{\rho}^P(x,y)$ for all neighboring particle configurations $x$ and $y$  related by $y = x - e_i + e_j$. We may choose any coupling in order to obtain an upper bound. The coupling will be synchronous: the choice of particle to be moved will be coordinated between the chains. Namely, if the ``extra'' particle is chosen in configuration $x$, then so too will the ``extra'' particle be chosen in configuration $y$. Similarly, if some other particle is chosen in $x$, than a particle at the same vertex will be chosen in $y$. For an illustration, see Figure \ref{fig:coupling}.

\begin{figure}[h]
\centering{
\begin{tikzpicture}
[every node/.style={}]
\draw (0,0) circle (0.4);
\draw (0,1) circle (0.4);
\draw (0,2) circle (0.4);
\draw (0,3) circle (0.4);
\draw (1.5,0) circle (0.4);
\draw (1.5,1) circle (0.4);
\draw (1.5,2) circle (0.4);
\draw (1.5,3) circle (0.4);
\draw[fill](0,0) circle (0.05);
\draw[fill](0,-0.2) circle (0.05);
\draw[fill](0,0.2) circle (0.05);
\draw[fill](1.5,0) circle (0.05);
\draw[fill](1.5,-0.2) circle (0.05);
\draw[fill](1.5,0.2) circle (0.05);
\draw[-](0,0) to (1.5,0);
\draw[-](0,0.2) to (1.5,0.2);
\draw[-](0,-0.2) to (1.5,-0.2);
\draw[fill](0,1.1) circle (0.05);
\draw[fill](0,0.9) circle (0.05);
\draw[fill](1.5,1.1) circle (0.05);
\draw[fill](1.5,0.9) circle (0.05);
\draw[-](0,1.1) to (1.5,1.1);
\draw[-](0,0.9) to (1.5,0.9);
\draw[fill](0,2) circle (0.05);
\draw[fill](1.5,2.1) circle (0.05);
\draw[fill](1.5,1.9) circle (0.05);
\draw[-](0,2) to (1.5,1.9);
\draw[fill](0,3) circle (0.05);
\draw[fill](0,2.8) circle (0.05);
\draw[fill](0,3.2) circle (0.05);
\draw[fill](1.5,2.9) circle (0.05);
\draw[fill](1.5,3.1) circle (0.05);
\draw[-](0,3.2) to (1.5,3.1);
\draw[-](0,3) to (1.5,2.9);
\draw[-](0,2.8) to (1.5,2.1);
\node at (-0.7,3){$i$};
\node at (-0.7,2){$j$};
\node at (0,-0.7){$x$};
\node at (1.5,-0.7){$y$};
\end{tikzpicture}
\caption{Pairing of particles in the coupling. The edges between vertices are omitted.}
\label{fig:coupling}}
\end{figure}
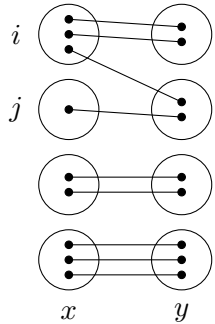

Let $X_1 \sim P(x, \cdot)$ and $Y_1 \sim P(y, \cdot)$ denote the coupled random variables corresponding to the next configurations. 
Let $p^{\star}$ be the ``extra'' particle. Let $\tilde{p}$ be a random variable that denotes the uniformly selected particle. Since our coupling gives an upper bound, we can write
\begin{align}
W_{\rho}^P(x,y) &\leq  \frac{1}{n} \mathbb{E} \left[\rho (X_1, Y_1) | \tilde{p} = p^{\star} \right] + \frac{n-1}{n} \mathbb{E} \left[\rho (X_1, Y_1) | \tilde{p} \neq p^{\star} \right].\label{eq:W-bound}
\end{align}

First, suppose the ``extra'' particle, $p^{\star}$, is chosen in both chains. This happens with probability $\frac{1}{n}$. By Lemma \ref{lemma:close-distributions}, we can couple the distributions $P_x(i, \cdot)$ and $P_y(j, \cdot)$ to $Q(i, \cdot)$ and $Q(j, \cdot)$ respectively with probability at least $1 -  \nicefrac{(e^{\nicefrac{\beta}{2}} - 1)}{(e^{\nicefrac{\beta}{2}} + 1)}$. In that case, we get contraction by a factor of $(1-\delta)$. With the remaining probability, we assume the worst-case distance of $d_{\text{max}}$. Therefore, the conditional Wasserstein distance is upper bounded as follows:
\begin{align}
 \mathbb{E} \left[\rho (X_1, Y_1) | \tilde{p} = p^{\star} \right]  &\leq  \left(1 -  \frac{e^{\nicefrac{\beta}{2}} - 1}{e^{\nicefrac{\beta}{2}} + 1}\right)(1-\delta) d(i,j) + \left(\frac{e^{\nicefrac{\beta}{2}} - 1}{e^{\nicefrac{\beta}{2}} + 1}\right) d_{\text{max}}. \label{eq:extra-particle}
\end{align}

Next, suppose some other particle (located at $v$) is chosen in both chains. This happens with probability $\frac{n-1}{n}$. We claim
\begin{align}
 \mathbb{E} \left[\rho (X_1, Y_1) | \tilde{p} \neq p^{\star} \right]  &\leq \rho(x,y) + 2d_{\text{max}}  \frac{\Delta + 1 }{n} \beta \label{eq:other-particle}.
\end{align}
Indeed, by Lemma \ref{lemma:TV-bound-same-vertex}, we can couple particle $\tilde{p}$ so that it moves to the same vertex in both chains with probability at least 
\[1 - \frac{\Delta + 1}{n} \beta.\] 
By Proposition \ref{prop:distance}, it holds that $\rho(X_1,Y_1) = d(i, j) = \rho(x,y)$ in the case that the particle $\tilde{p}$ moves to the same vertex in both chains. Otherwise, an additional distance of at most $2 d_{\text{max}}$ is incurred.

Finally, we substitute the bounds \eqref{eq:extra-particle} and \eqref{eq:other-particle} into \eqref{eq:W-bound}.
\small
\begin{align}
&W_{\rho}^P (x, y) \nonumber \\
&\leq \rho(x,y) + \frac{1}{n} \left(- \rho(x,y) + \left(1 -  \frac{e^{\nicefrac{\beta}{2}} - 1}{e^{\nicefrac{\beta}{2}} + 1}\right)(1-\delta) \rho(x,y) + \left(\frac{e^{\nicefrac{\beta}{2}} - 1}{e^{\nicefrac{\beta}{2}} + 1}\right) d_{\text{max}}  \right) +\frac{n-1}{n} \left(\frac{\Delta + 1 }{n} 2 \beta  d_\text{max} \right) \nonumber\\
&= \rho(x,y) \left[1 - \frac{1}{n}\left(1 - \left(1 -  \frac{e^{\nicefrac{\beta}{2}} - 1}{e^{\nicefrac{\beta}{2}} + 1}\right)(1-\delta) -  \left(\frac{e^{\nicefrac{\beta}{2}} - 1}{e^{\nicefrac{\beta}{2}} + 1}\right) \frac{d_{\text{max}}}{\rho(x,y)} - \frac{n-1}{n} \left(\frac{\Delta + 1 }{ \rho(x,y)} 2\beta d_\text{max}  \right) \right) \right] \nonumber\\
&\leq \rho(x,y) \left[1 - \frac{1}{n}\left(1 - \left(1 -  \frac{e^{\nicefrac{\beta}{2}} - 1}{e^{\nicefrac{\beta}{2}} + 1}\right)(1-\delta) -  \left(\frac{e^{\nicefrac{\beta}{2}} - 1}{e^{\nicefrac{\beta}{2}} + 1}\right) d_{\text{max}} -  (\Delta + 1) 2\beta d_\text{max} \right) \right] \label{eq:W-bound-2}
\end{align}
\normalsize
where the last inequality is due to $\rho(x,y) \geq 1$ and $\frac{n-1}{n} < 1$. In order to show contraction, it is sufficient that the expression multiplying $\frac{1}{n}$ be positive: \begin{align*}
&1 - \left(1 -  \frac{e^{\nicefrac{\beta}{2}} - 1}{e^{\nicefrac{\beta}{2}} + 1}\right)(1-\delta) -  \left(\frac{e^{\nicefrac{\beta}{2}} - 1}{e^{\nicefrac{\beta}{2}} + 1}\right) d_{\text{max}} -(\Delta + 1)2\beta d_\text{max}  > 0\\
\iff &\frac{e^{\nicefrac{\beta}{2}} - 1}{e^{\nicefrac{\beta}{2}} + 1}  < \frac{ \delta - (\Delta + 1)2 \beta  d_\text{max} }{d_{\text{max}} + \delta - 1}.
\end{align*}
For an example of a satisfying $\beta$, choose $\beta$ so that 
\[(\Delta + 1)2 \beta  d_\text{max} < \frac{\delta}{2} \text{~~and~~} \frac{e^{\nicefrac{\beta}{2}} - 1}{e^{\nicefrac{\beta}{2}} + 1} = \tanh\left( \frac{\beta}{4} \right)< \frac{\delta}{2 \left(d_{\text{max}} + \delta - 1 \right)}.\]
Therefore, we can choose 
$$0 < \beta_- < \min \left \{ \frac{\delta}{4 d_\text{max} (\Delta + 1) }, 4 \tanh^{-1} \left(\frac{\delta}{2 \left(d_{\text{max}} + \delta - 1 \right)} \right) \right \}.$$
Substituting $\beta = \beta_-$ into \eqref{eq:W-bound-2}, we obtain for some $\delta' > 0$
\begin{align*}
W_{\rho}^P(x,y) &\leq \rho(x,y) \left(1 - \frac{1}{n} \delta'\right).
\end{align*}
Applying the path coupling lemma (Lemma \ref{lemma:path-coupling}), we obtain
$$d(X,t) \leq \left(1 - \frac{1}{n} \delta'\right)^t diam(\Omega) \leq \left(1 - \frac{1}{n} \delta'\right)^t n d_{\text{max}}.$$
Setting the right hand side to be less than $\epsilon > 0$ in order to bound $t_{\text{mix}}(X, \epsilon)$,
\begin{align*}
\left(1 - \frac{1}{n} \delta'\right)^t n d_{\text{max}} \leq \epsilon \iff  t \geq \frac{\log \left( \frac{\epsilon}{n d_{\text{max}}} \right)}{\log \left(1 - \frac{1}{n} \delta' \right)} \iff  t \geq \frac{\log \left( \frac{n d_{\text{max}}}{\epsilon} \right)}{\log \left( \frac{n}{n - \delta'}\right)}.
\end{align*}
Since 
\[\log \left( \frac{n}{n - \delta'}\right) = \log \left( 1 + \frac{\delta'}{n - \delta'}\right) \geq \frac{\frac{\delta'}{n - \delta'}}{1+ \frac{\delta'}{n - \delta'}},\]
we have 
\begin{align*}
\frac{\log \left( \frac{n d_{\text{max}}}{\epsilon} \right)}{\log \left( \frac{n}{n - \delta'}\right)} \leq \log \left( \frac{n d_{\text{max}}}{\epsilon} \right) \left(1 + \frac{\delta'}{n - \delta'} \right) \frac{n- \delta'}{\delta'} = O(n \log n).
\end{align*}
Therefore, $t_{\text{mix}}(X, \epsilon) = O(n \log n)$, which completes the proof of Theorem \ref{thm:fast-mixing}.
\end{proof}

\begin{remark}
Arguably, a more natural approach to show fast mixing would be through a more traditional path coupling approach:  Let $\mathcal{H}$ have an edge between configurations $x$ and $y = x - e_i + e_j$ if $i$ and $j$ are adjacent vertices in $\mathcal{G}$. Set $l(x,y) = 1$ for adjacent configurations. However, this approach does not yield contraction in the Wasserstein distance, which we show at the end of this section. 
\end{remark}

We now provide the deferred proofs.
\begin{proof}[Proof of Lemma \ref{single-RW}]
First we verify that $d(x,y)$ is a metric. It holds that $d(x,y) = d(y,x)$, and $d(x,y) \geq 0$ with equality if and
only if $x = y$. To show the triangle inequality, start three random walks from vertices $x,y,z$ and let $\tau(x,y)$ be the
meeting time of the walks started from $x$ and $y$. The three random walks are advanced according to the independent coupling, and if a pair of walks collides, they are advanced identically starting from that time. Under this coupling, observe that $$\tau(x,z) \leq \max \{\tau(x,y), \tau(y,z) \} \leq \tau(x,y) + \tau(y,z)$$ and take expectations. Next we show that $W_{\rho}^{Q}(x,y) \leq  d(x,y) - 1$ for $x \neq y$. We can choose any coupling of $X_1 \sim P(x,\cdot)$ and $Y_1 \sim P(y, \cdot)$ to show an upper bound. Letting $X_1 \sim P(x,\cdot)$ and $Y_1 \sim P(y, \cdot)$ be independent, we have 
$$W_{\rho}^{Q}(x,y) \leq \mathbb{E} \left[ \tau(X_1, Y_1) \right] = \sum_{a,b} Q(x,a)Q(y,b) \mathbb{E}[\tau(a,b)]$$
and
$$d(x,y) = \mathbb{E}[\tau(x,y)] = 1 + \sum_{a,b} Q(x,a) Q(y,b) \mathbb{E}[\tau(a,b)].$$
These two equations imply $W_{\rho}^{Q}(x,y) \leq d(x,y) - 1$.
Finally, $d(x,y) - 1 \leq d(x,y) \left(1 - \nicefrac{1}{d_{\text{max}}}\right)$.
If $x \sim y$, then we conclude $d(x,y) - 1 \leq d(x,y) \left(1 - \nicefrac{1}{d'_{\text{max}}}\right)$.
\end{proof}

\begin{proof}[Proof of Proposition \ref{prop:distance}]
Consider any path from $x$ to $y$: $(x = x_0, x_1, \dots, x_{m-1}, x_m = y)$, where $x_{r+1} = x_r - e_{i_r} + e_{j_r}$ for $r \in \{0, 1, \dots, m-1\}$. Then we have
\begin{align*}
\sum_{r = 0}^{m-1} l(x_r, x_{r+1}) &= \sum_{r = 0}^{m-1} d(i_r, j_r).
\end{align*}
We claim that we can rearrange this summation to be of the form 
$$d(i, l_1) + \sum_{r = 1}^{m-2} d(l_r, l_{r+1}) + d(l_{m-1}, j)$$
for some sequence $l_1, \dots, l_{m-1}$. Indeed, let $\mathcal{I} = \{i_r : 0 \leq r \leq m-1\}$ and $\mathcal{J} = \{j_r : 0 \leq r \leq m-1\}$ be the multisets that collect the ``outbound'' and ``inbound'' particle transfers, respectively. The value $i$ must appear one more time in $\mathcal{I}$ than in $\mathcal{J}$. Similarly, the value $j$ must appear one more time in $\mathcal{J}$ than in $\mathcal{I}$. All other values appear an equal number of times in $\mathcal{I}$ and $\mathcal{J}$. By choosing terms $d(i_r, j_r)$ in order, beginning with $d(i, l_1)$, it is possible to rearrange the sum into the given form.
By the triangle inequality for $d(\cdot, \cdot)$,
\begin{align*}
d(i, l_1) + \sum_{r = 1}^{m-2} d(l_r, l_{r+1}) + d(l_{m-1}, j) &\geq d(i, j) = l(x,y).
\end{align*}
Therefore, the shortest distance between $x$ and $y$ is along the edge connecting them, and we conclude that $\rho(x,y) = l(x,y)$ for neighboring configurations.
\end{proof}

To prove Lemma \ref{lemma:close-distributions}, we state the following proposition.

\begin{proposition}\label{prop:convex}
The set of distributions $\{P_x(i, \cdot) : x \in \Omega\}$ parametrized by the configuration $x$ is contained within the convex set $$P_{\beta} \triangleq \left \{(p_0, \dots, p_d) : \frac{p_i}{p_j} \leq e^{\beta} ~\forall i, j; \sum_{i = 0}^d p_i = 1; p_i \geq 0 ~\forall i \right \}.$$ 
\end{proposition}
\begin{proof}
To show this claim, we compute the ratio $\frac{P_x(i, j_1)}{P_x(i, j_2)}$  when $j_1, j_2 \in \mathcal{N}(i) \cup \{i\}$ and $j_1 \neq j_2$, and show that it is upper bounded by $e^{\beta}$. There are three cases to consider.
\begin{enumerate}
\item The case $j_1 = i$. 
\begin{align*}
\frac{P_x(i, j_1)}{P_x(i, j_2)} &= \frac{\f{(x(j_1) - 1)}}{\f{x(j_2)}} = \f{\left( x(j_1) - x(j_2) - 1 \right)}.
\end{align*}
Since $x(j_1) - x(j_2) - 1 \leq n -1 < n$, it holds that $\frac{P_x(i, j_1)}{P_x(i, j_2)} < e^{\beta}$.
\item The case $j_2 = i$.
\begin{align*}
\frac{P_x(i, j_1)}{P_x(i, j_2)} &= \f{\left( x(j_1) - x(j_2) + 1 \right)}.
\end{align*}
Since $j_2 = i$, we have $j_2 \geq 1$. Therefore, again $\frac{P_x(i, j_1)}{P_x(i, j_2)} < e^{\beta}$.
\item The case $j_1, j_2 \neq i$.
\begin{align*}
\frac{P_x(i, j_1)}{P_x(i, j_2)} &= \f{\left( x(j_1) - x(j_2) \right)} \leq e^{\beta}. \qedhere
\end{align*}
\end{enumerate}
\end{proof}

\begin{proof}[Proof of Lemma \ref{lemma:close-distributions}]
Recall that $\mathcal{N}(i)$ is the neighbor set of vertex $i$ in graph $\mathcal{G}$. Let $d = \left| \mathcal{N}(i) \right|$. We have
$$P_x(i,j) = \begin{cases} \frac{\exp\left(\frac{\beta}{n} x(j)\right)}{\sum_{l \sim i} \exp\left(\frac{\beta}{n} x(l)\right) + \exp \left(\frac{\beta}{n} \left(x(i)-1 \right)\right)} & \text{ if } i \sim j\\
 \frac{\exp\left(\frac{\beta}{n} \left(x(i) - 1 \right)\right)}{\sum_{l \sim i} \exp\left(\frac{\beta}{n} x(l)\right) + \exp \left(\frac{\beta}{n} \left(x(i)-1 \right)\right)} & \text{ if } i = j\\
0 & \text{ otherwise }
\end{cases}
$$
and
$$
Q(i, j) = \begin{cases} 
\frac{1}{d + 1} & \text{ if } i \sim j\\
\frac{1}{d + 1} & \text{ if } i = j\\
0 & \text{ otherwise. }
\end{cases}$$
Using Proposition \ref{prop:convex},
\begin{align}
\max_x \left \Vert P_x(i, \cdot) - Q(i, \cdot) \right \Vert_{\text{TV}} &\leq \sup_{p \in P_{\beta}} \left \Vert p - Q(i, \cdot) \right \Vert_{\text{TV}} = \frac{e^{\beta}}{d + e^{\beta}} - \frac{1}{d+1}. \label{eq:convex}
\end{align} 
The inequality is due to the fact that $\{P_x(i, \cdot) : x \in \Omega \} \subset P_{\beta}$ and the equality is due to the fact that the maximum of a convex function over a closed and bounded convex set is achieved at an extreme point, namely $\left( \frac{e^{\beta}}{d + e^{\beta}}, \frac{1}{d + e^{\beta}}, \dots, \frac{1}{d + e^{\beta}} \right)$. To maximize the right hand side of \eqref{eq:convex}, let $f(d) = \frac{e^{\beta}}{d + e^{\beta}} - \frac{1}{d+1}$. Then 
\begin{align*}
f'(d) &=- \frac{e^{\beta}}{\left(d + e^{\beta}\right)^2} + \frac{1}{(d+1)^2}\\
&= \frac{\left(d + e^{\beta}\right)^2 - e^{\beta} \left(d+1\right)^2}{\left(d + e^{\beta}\right)^2 \left(d+1\right)^2}\\
&=  \frac{\left(d + e^{\beta} - e^{\nicefrac{\beta}{2}}(d+1) \right)\left(d + e^{\beta} + e^{\nicefrac{\beta}{2}}(d+1) \right)}{\left(d + e^{\beta}\right)^2 \left(d+1\right)^2}.
\end{align*}
Setting $f'(d) = 0$ we obtain the solutions $d = \pm e^{\nicefrac{\beta}{2}}$. The solution $d =  e^{\nicefrac{\beta}{2}}$ is the maximizer. Substituting $d = e^{\nicefrac{\beta}{2}}$ into \eqref{eq:convex},
\begin{align*}
\max_{x, i} \|P_x(i, \cdot) - Q(i,\cdot)\|_{TV} &\leq \frac{e^{\beta}}{e^{\nicefrac{\beta}{2}} + e^{\beta}} - \frac{1}{e^{\nicefrac{\beta}{2}} + 1} =  \frac{e^{\nicefrac{\beta}{2}}}{e^{\nicefrac{\beta}{2}} + 1} - \frac{1}{e^{\nicefrac{\beta}{2}} + 1} =  \frac{e^{\nicefrac{\beta}{2}} - 1}{e^{\nicefrac{\beta}{2}} + 1},
\end{align*}
which completes the proof.
\end{proof}

\begin{proof}[Proof of Lemma \ref{lemma:TV-bound-same-vertex}]
First, 
\begin{align*}
\left \Vert P_x(v, \cdot), P_y(v, \cdot) \right \Vert_{\text{TV}} &= \frac{1}{2} \sum_{w \in \mathcal{N}(v) \cup \{v\}} \left | P_x(v, w) - P_y(v, w) \right|.
\end{align*}
We will show that each term is upper bounded by $\frac{2\beta}{n}$. Since there are at most $\Delta + 1$ terms, the bound follows.

We compute $\max_{x, y : x \sim y}  \left | P_x(v, w) - P_y(v, w) \right|$ for $w \in \mathcal{N}(v) \cup \{v\}$. Since $x$ and $y$ are interchangeable, we can drop the absolute value.
\begin{align*}
\max_{x, y : x \sim y}  \left | P_x(v, w) - P_y(v, w) \right| &=  \max_{x, y : x \sim y}  P_x(v, w) - P_y(v, w).
\end{align*}

First consider the case that $v \neq w$. Then 
\begin{align*}
&\max_{x, y : x \sim y}  P_x(v, w) - P_y(v, w) \\
&=  \max_{x, y : x \sim y} \frac{\f{x(w)}}{\f{(x(v)-1)} + \sum_{u \sim v} \f{x(u)}} - \frac{\f{y(w)}}{\f{(y(v)-1)} + \sum_{u \sim v} \f{y(u)}}.
\end{align*}
Let
\[A(z) = \f{(z(v)-1)} + \sum_{u \sim v} \f{z(u)}.\]
Note that $A(y) \leq e^{\nicefrac{\beta}{n}} A(x)$ for $x \sim y$. We have
\begin{align*}
\max_{x, y : x \sim y}  P_x(v, w) - P_y(v, w) &=\max_{x, y : x \sim y}   \frac{\f{x(w)}}{A(x)} - \frac{\f{y(w)}}{A(y)}\\
&\leq\max_{x, y : x \sim y}  \frac{\f{x(w)}}{A(x)} - \frac{\f{y(w)}}{e^{\nicefrac{\beta}{n}} A(x)}\\
&= \max_{x, y : x \sim y} \frac{\f{x(w)} - \f{(y(w) -1)}}{A(x)}\\
&\leq  \max_{x} \frac{\f{x(w)} - \f{(x(w) -2)}}{A(x)}\\
&= \max_x \frac{\f{(x(w)} \left(1 - e^{-\nicefrac{2\beta}{n}}\right)}{A(x)}\\
&\leq 1 - e^{-\nicefrac{2\beta}{n}}\\
&\leq \frac{2\beta}{n}.
\end{align*}
Next, we consider the case $v = w$. We have
\begin{align*}
\max_{x, y : x \sim y}  P_x(v, v) - P_y(v, v) &=  \max_{x, y : x \sim y} \frac{\f{(x(v)-1)}}{A(x)} - \frac{\f{(y(v)-1)}}{A(y)}  \\
&\leq\max_{x, y : x \sim y}  \frac{\f{(x(v)-1)}}{A(x)} - \frac{\f{(y(v)-1)}}{e^{\nicefrac{\beta}{n}} A(x)}\\
&= \max_{x, y : x \sim y} \frac{\f{(x(v)-1)} - \f{(y(v) -2)}}{A(x)}\\
&\leq  \max_{x} \frac{\f{(x(v)-1)} - \f{(x(v) -3)}}{A(x)}\\
&= \max_x \frac{\f{((x(v)-1)} \left(1 - e^{-\nicefrac{2\beta}{n}}\right)}{A(x)}\\
&\leq 1 - e^{-\nicefrac{2\beta}{n}}\\
&\leq \frac{2\beta}{n}. \qedhere
\end{align*}
\end{proof}


We now show that the approach for proving Theorem \ref{thm:fast-mixing} based on the natural one-step path coupling does not yield the required contraction.
\begin{theorem}\label{thm:no-contraction}
Let $\mathcal{H}$ have an edge between configurations $x$ and $y = x - e_i + e_j$ whenever $i$ and $j$ are adjacent vertices in $\mathcal{G}$. Let $l(x,y) = 1$ for adjacent configurations. There exists a graph $\mathcal{G}$ such that for $\beta = 0$, 
\begin{align*}
W_{\rho}^P(x,y) \geq 1
\end{align*}
for some adjacent configurations $x, y$.
\end{theorem}
\begin{proof}
Let $\mathcal{G}$ be the 4-vertex path graph. Label the vertices $1,2,3,4$ in order along the path, and consider $x$ and $y$ related by $y = x -e_2 + e_3$ so that the two configurations differ by a transfer from one middle vertex to the other. When $\beta = 0$, the transition probabilities are simple: given that a particle is chosen at vertex $v$, it moves to vertex $w \in \mathcal{N}(v) \cup \{v\}$ with probability $\frac{1}{deg(v) + 1}$. The optimal coupling of $P(x, \cdot)$ and $P(y, \cdot)$ may be expressed as an optimal solution of a linear program, as follows. Write $x' \sim x$ if $x'$ is adjacent to $x$ in $\mathcal{H}$ or $x' = x$. For each $x' \sim x$ and $y' \sim y$, let $z(x',y')$ be the probability of the next states being $x'$ and $y'$ in a coupling. The constraints require the collection of $z$ variables to be a valid coupling, and the objective function calculates the expected distance under the coupling.
\begin{align*}
\min & \sum_{x' \sim x, y' \sim y} z(x', y') \rho(x',y')\\
\text{s.t. }  &\sum_{y' \sim y} z(x', y') = P(x,x') & \forall x' \sim x\\
&\sum_{x' \sim x} z(x', y') = P(y,y') &\forall y' \sim y\\
&z \geq 0
\end{align*}
This linear program is known as a Kantorovich problem. Our goal is to show that the optimal objective value is at least $1$. We will first write down the dual problem. By weak duality, any feasible solution to the dual problem gives a lower bound to the optimal value of the primal problem. Next we will construct a primal solution with objective value equal to $1$, and apply the complimentary slackness condition to help us construct a dual solution whose objective value is also equal to $1$. Finally we will conclude that the optimal solution to the primal problem is equal to $1$, by strong duality. For a reference to linear programming duality, see e.g. Chapter 4 of \cite{Bertsimas1997}.

First we take the dual of the linear program, introducing dual variables $u(x')$ for $x' \sim x$ and $v(y')$ for $y' \sim y$:
\begin{align*}
\max & \sum_{x' \sim x} u(x')P(x,x') + \sum_{y' \sim y} v(y') P(y,y')\\
\text{s.t. }  &u(x') + v(y') \leq \rho(x', y') &\forall x' \sim x, y' \sim y
\end{align*}
This linear program is a Kantorovich dual problem. By weak duality, if there exists a dual solution with objective value $Z$, then the optimal solution of the primal is at least $Z$. Therefore our goal is to find a dual solution with objective value at least $1$. 

For $x' = x - e_a + e_b$ with $a,b \in \{1,2,3,4\}$, $P(x,x') = \frac{x(a)}{n \left(\deg(a)+1\right)}$. Similarly, for $y' = y - e_a + e_b$, $P(y, y') = \frac{y(a)}{n \left(\deg(a)+1\right)}$. The value of $\rho$ is given by
\begin{align*}
\rho(x', y') &= \begin{cases}
0 & \text{if } [x' = y' = x] \text{ or } [x' = y' = y]\\
1 & \text {if } [x' = y, y' \neq y] \text{ or } [y' = x, x' \neq x] \\
1 & \text{if } [x' = x - e_a + e_b, y' = y -e_a + e_b, a\neq b]\\
2 & \text{if }[x' = x, y' \notin \{x,y\}] \text{ or } [y' = y, x' \notin \{x,y\}]\\
3 & \text{otherwise}
\end{cases}.
\end{align*}


There exists a primal solution with objective value $1$: Set 
$$z(x-e_a+e_b, y-e_a + e_b) = \min \left\{P(x,x-e_a+e_b), P(y,y-e_a+e_b)\right\},$$ 
$$z(x-e_2 + e_1, y-e_3+e_2) = \frac{1}{3n},$$
and $$z(x-e_2+e_3, y-e_3+e_4) = \frac{1}{3n}.$$ Other values of $z(x', y')$ are set to zero. In other words, $z$ describes a synchronous coupling according to the pairing in Figure \ref{fig:coupling}, with particles moving in the same direction always. Now supposing this is an optimal solution, we apply complementary slackness to identify candidate dual optimal solutions. The complementary slackness condition states that if $z$ and $(u,v)$ are optimal primal and dual solutions, then it holds that for all $x' \sim x, y' \sim y$, $$z(x', y') \left[\rho(x', y') - u(x') - v(y')\right] = 0.$$ If our primal solution $z$ is optimal, then whenever $z(x',y') \neq 0$, we need $u(x') + v(y') = \rho(x', y')$. These additional constraints help us construct the following dual feasible solution: 
\begin{align*}
u(x) = 1, u(x-e_1+e_2) = 0, u(x-e_2+e_1) = 2, u(x-e_2+e_3) = 0,\\
u(x-e_3+e_2)=2, u(x-e_3+e_4)=0, u(x-e_4+e_3) = 0\\
v(y) = 0, v(y-e_1+e_2) = 1, v(y-e_2+e_1) = -1, v(y-e_2+e_3) = 1,\\
v(y-e_3+e_2)=-1, v(y-e_3+e_4)=1, v(y-e_4+e_3) = 1.
\end{align*}
We find that the objective value of this solution is equal to $1$.
By strong duality, we conclude that the optimal value of the primal problem is equal to $1$, and therefore there does not exist a contractive coupling.  
\end{proof}
\begin{remark}
The argument in the proof of Theorem \ref{thm:no-contraction} should apply to all graphs $\mathcal{G}$ that contain the a four-vertex path graph as a subgraph, and possibly to other graphs as well.
\end{remark}

\subsection{Bounding the Cheeger constant}\label{sec:cheeger}
The following results will be useful in proving Corollary \ref{corollary:cheeger}. \begin{lemma}[\cite{Peres2017}]\label{lemma:cheeger-1}
Let $\Phi_*$ be the Cheeger constant of an aperiodic and irreducible Markov chain, and let $t_{\text{mix}}$ be its mixing time. Then
\[t_{\text{mix}}\left(\frac{1}{4}\right) \geq \frac{1}{4 \Phi_*}.\]
\end{lemma}

The following result follows directly from Equation 2.28 in \cite{Fill1991}.
\begin{lemma}\label{lemma:cheeger-2}
Let $P$ be the transition probability matrix of an aperiodic and irreducible Markov chain on state space $\mathcal{X}$, satisfying $P(x,x) \geq \frac{1}{2}$ for all $x \in \mathcal{X}$. Let $\pi$ denote the stationary distribution, and let $\pi_{\text{min}} = \min_{x \in \mathcal{X}} \pi(x)$. Let $\Phi_*$ be the Cheeger constant of this chain, and let $t_{\text{mix}}$ be its mixing time. Then
\begin{align*}
t_{\text{mix}}(\epsilon) \leq \frac{2 \log \left( \frac{1}{2 \epsilon \sqrt{\pi_{\text{min}}}}\right)}{\log \left(\frac{2}{2 - \Phi_*^2} \right)}.
\end{align*}
\end{lemma}

\begin{proof}[Proof of Corollary \ref{corollary:cheeger}]
We first prove the lower bound. Let $0 \leq \beta < \beta_-$. By Lemma \ref{lemma:cheeger-1} and Theorem \ref{thm:fast-mixing}, we have
\[\frac{1}{4 \Phi_*} = O\left( n \log n\right)  \implies \Phi_* = \frac{1}{O\left( n \log n\right)}.\]
We next prove the upper bound. Let $\overline{P} = \frac{1}{2}(P + I)$ denote the lazy version of the ARW chain. Note that $\pi$ is also the stationary distribution of the lazy chain.  Let $\overline{\Phi}$ denote its Cheeger constant. Observe that for all $S \subset \Omega$, it holds that $\Phi(S) = 2 \overline{\Phi}(S)$. Therefore,
\[\Phi_* = \min_{S : \pi(S) \leq \frac{1}{2}} \Phi(S) = 2\min_{S : \pi(S) \leq \frac{1}{2}} \overline{\Phi}(S) = 2 \overline{\Phi}_*,\]
and so it suffices to show $\overline{\Phi}_* = e^{-\Omega(n)}$. 

Fix $\beta > \beta^+$. We claim that the mixing time of the lazy chain is $e^{\Omega(n)}$. To show this, we modify the proof of Theorem \ref{thm:slow-mixing} as follows. We define $\overline{Z}$ to be the lazy version of the $Z$ chain. That way, we can couple the $\overline{Z}$ chain to the chain advancing according to $\overline{P}$. Observe that the lower-bounding property still holds. Furthermore, the hitting time of the $\overline{Z}$ chain to the set $\{x \in \tilde{\Omega}: x(0) \leq (1-\delta) n\}$ is greater than the hitting time for the $Z$ chain. Finally, since the stationary distributions of the $\overline{Z}$ and $Z$ chains coincide, the rest of the proof follows identically.  

We now apply Lemma \ref{lemma:cheeger-2} to the lazy chain. We need a lower bound on $\pi_{\text{min}}$. Observe that if $t \in \mathbb{N}$ and $p \in [0,1]$ are such that $P^t(x,y) \geq p$ for all $x, y \in \Omega$, then $\pi_{\text{min}} \geq p$. Set $t = n  \cdot diam(\mathcal{G})$. There exists at least one sequence of possible transitions in $t$ steps to get from $x$ to $y$. Each step in the sequence has probability at least $\frac{1}{n(\Delta + e^{\beta})}$. Therefore, 
\[\pi_{\text{min}} \geq  \left(\frac{1}{n(\Delta + e^{\beta})} \right)^{n  \cdot diam(\mathcal{G})}.\]
By Lemma \ref{lemma:cheeger-2}, we have 
\[ \frac{2 \log \left( \frac{1}{2 \epsilon \sqrt{\pi_{\text{min}}}}\right)}{\log \left(\frac{2}{2 - \overline{\Phi}_*^2} \right)} = e^{\Omega(n)}.\]
Since $\log(x) \geq 1 - \frac{1}{x}$ for $x > 0$, we have
\begin{align*}
&\frac{2 \log \left( \frac{1}{2 \epsilon \sqrt{\pi_{\text{min}}}}\right)}{1 - \frac{2 - \overline{\Phi}_*^2}{2}} = e^{\Omega(n)}\\
&\overline{\Phi}_*^2 = 4e^{-\Omega(n)}  \log \left( \frac{1}{2 \epsilon \sqrt{\pi_{\text{min}}}}\right)\\
&\overline{\Phi}_*^2 = e^{-\Omega(n)} \log\left(\frac{1}{\pi_{\text{min}} }\right).
\end{align*}
Substituting the lower bound for $\pi_{\text{min}}$, we obtain
\begin{align*}
\overline{\Phi}_*^2 =  e^{-\Omega(n)} \log\left(\left(2n(\Delta + e^{\beta}) \right)^{n  \cdot diam(\mathcal{G})}   \right) =  n \log(n) e^{-\Omega(n)}  = e^{-\Omega(n)},
\end{align*}
which implies $\overline{\Phi}_* = e^{-\Omega(n)}$. 
\end{proof}

\section{Repelling Random Walks}\label{sec:repelling}
Throughout our analysis, we have only considered $\beta \geq 0$. However, the case $\beta < 0$ (``Repelling Random Walks'') is theoretically and practically interesting to study also. Simulations confirm the intuition that the particles behave like independent random walks when $\beta$ is close to zero, and spread evenly when $\beta$ is very negative (see Figure \ref{fig:negative-beta}). We conjecture that there are not any hard-to-escape subsets of the state space for all $\beta < 0$.

\begin{figure}[h]
\centering
\includegraphics[width=0.4\textwidth]{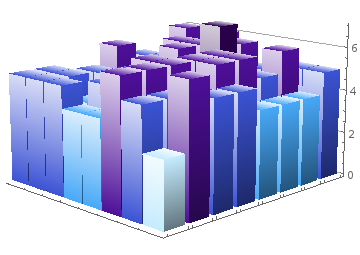}
\caption{Simulation of the Attracting Random Walks model on an $8 \times 8$ grid graph after $10^6$ steps for $n = 320$, $\beta = -500$.}
\label{fig:negative-beta}
\end{figure}

\begin{conjecture}
For all $\beta < 0$ and any graph, the mixing time of the ARW model is polynomial in $n$.
\end{conjecture}

We consider two cases: the extreme case of $\beta = - \infty$, and the case where $\mathcal{G}$ is the complete graph, for certain values of $\beta$. 

\subsection{The Case $\beta = - \infty$}
\begin{theorem}
When $\beta = -\infty$, the mixing time of the Attracting Random Walks model is $O(n^2)$.
\end{theorem}
\begin{proof}
When $\beta = -\infty$, the dynamics are simplified. Suppose a particle is chosen at vertex $i$. Let $A$ be the set of vertices corresponding to the minimal value(s) of $\{x(i) -1\} \cup\{ x(j) : j \sim i\}$. The chosen particle moves to a vertex among those in $A$, uniformly at random. 

Our goal is to show that the set $$C \triangleq \left \{x : x(v) \in \left \{ \left \lfloor \frac{n}{k} \right \rfloor, \left \lfloor \frac{n}{k} \right \rfloor + 1 \right \} \forall v \in V \right\} \cap \Omega$$ satisfies the following three properties: (1) It is absorbing, meaning that once the chain enters $C$, it cannot escape $C$; (2) The chain enters $C$ in polynomial time; (3) Within $C$, the chain mixes in constant time with respect to $n$.

We claim that the maximum particle occupancy cannot increase, and the minimum particle occupancy cannot decrease. We now show that the maximum particle occupancy, $M_{t} \triangleq \max_{v} X_t(v)$, is monotonically non-increasing over time. Suppose that at time $t$, a particle at vertex $i$ is selected and moves to vertex $j$. There are five cases:
\begin{enumerate}
    \item $i = j$. The maximum does not change.
    \item $i \neq j$, and both are maximizers. This case is not possible, since $x(j) > x(i) -1$.
    \item $i \neq j$, $i$ is a maximizer, and $j$ is not. The new maximum value is at most $M_t$, in the case that $X_t(j) = X_t(i) - 1$.
    \item $i \neq j$, $i$ is not a maximizer, and $j$ is. This case is not possible, since $x(j) > x(i) -1$.
    \item $i \neq j$, $i$ and $j$ are not maximizers. The new maximum value is at most $M_t$, in the case that $X_t(j) = X_t(i) - 1$.
\end{enumerate}
Therefore $M_{t+1} \leq M_t$. A similar argument shows that the minimum particle occupancy is monotonically non-decreasing over time. Together, they imply Property (1).

Next, we show Property (2). Assume $X_t \notin C$. Let $\mathcal{M}_t$ be the set of maximizing vertices at time $t$. We claim there exists at least one vertex $u \in \mathcal{M}_t$ such that there exists a path of distinct vertices $u = i_1 \sim i_2 \sim  \dots \sim i_p$ satisfying $x_{i_2} = x_{i_3} = \dots = x_{i_{p-1}} = M_t - 1$ and $x_{i_p} \leq M_t - 2$ (allowing $p = 2$). In other words, there is a walkable path from $u = i_1$ to $i_p$. The maximum length of the path is $k-1$. The probability that a particle is transferred along this path before any other events happen is therefore lower bounded by 
\[\left(\frac{M_t}{n}\cdot \frac{1}{\Delta+1}\right)^{k-1} \geq \left(\frac{1}{k}\cdot \frac{1}{\Delta+1}\right)^{k-1}.\]
Therefore the probability that such a transfer happens within $T_1$ trials is at least 
\[p \triangleq 1 - \left(1 - \left(\frac{1}{k}\cdot \frac{1}{\Delta+1}\right)^{k-1} \right)^{T_1}.\] If there had been at least two maximizing vertices to start, the number of maximizing vertices would have fallen by $1$. If there had been only one maximizing vertex to start, the maximum value itself would have fallen by $1$. 

We see that there are two types of ``good'' events: reducing the number of maximizing vertices while the maximum value stays the same, or reducing the maximum value. We claim that the number of ``good'' events that happen before the chain enters the set $C$ is upper bounded by $ n^2$. Indeed, imagine that the particles at each vertex are stacked vertically. A particle movement from vertex $i$ to vertex $j$ is interpreted as a particle moving from the top of the stack at vertex $i$ to the top of the stack at vertex $j$. Observe that the height of a particle cannot increase. Further, each particle's height can fall by at most $n-1$ units over time, and can therefore drop at most $n-1$ times. Since all good events require a particle's height to drop, the number of good events is at most $n(n - 1) < n^2$.  Let $T_2 = \lceil 2n^2\frac{1}{p} \rceil$ be the number of trials of length $T_1$ each. Let $N$ be the number of successes during the $T_2$ trials. By the Hoeffding inequality, 
\begin{align*}
 & \mathbb{P} \left(\left|N - \mathbb{E}[N] \right| \geq n^2 \right) \leq 2 \exp \left( - \frac{2 n^4}{T_2}\right) \leq 2 \exp \left( - \frac{2 n^4}{2n^2 \frac{1}{p} + 1}\right).
\end{align*}
Since $\mathbb{E}[N] = p  \lceil 2n^2\frac{1}{p} \rceil \geq 2n^2 $, 
\begin{align*}
 \mathbb{P} \left(N \leq n^2 \right)  \leq 2 \exp \left( - \frac{2 n^4}{2n^2 \frac{1}{p} + 1}\right) \leq 2 \exp \left(- \frac{1}{2}p n^2\right).
\end{align*}

Therefore the probability that the chain is in $C$ after $T_1 \times (k-1) \times T_2$ steps is at least $ 1 - 2\exp(-\frac{1}{2}pn^2)$. For an example, we can even set $T_1 = 1$. Then 
\[p = \left(\frac{1}{k}\cdot \frac{1}{\Delta+1}\right)^{k-1} \text{~~and~~}T_2 \leq 1 + 2n^2 \left(\frac{1}{k}\cdot \frac{1}{\Delta+1}\right)^{1-k},\] 
Therefore, within $O(n^2)$ steps, the chain is in $C$ with high probability.

Finally, we show Property (3). Once the chain is in $C$, there are two types of vertices: those that have $\left \lfloor \frac{n}{k} \right \rfloor$ particles, and those that have $\left \lfloor \frac{n}{k} \right \rfloor + 1$ particles. Note that there are always $\tilde{k} \triangleq n - k\lfloor \frac{n}{k} \rfloor $ vertices with the higher number of particles. Therefore it is equivalent to study an exclusion process with just $\tilde{k}$ particles on the graph $\mathcal{G}$. With probability $\left \lfloor \frac{n}{k} \right \rfloor \cdot \frac{k-\tilde{k}}{n}$, an unoccupied vertex is selected, and the chain stays in place. With the remaining probability, an occupied vertex is chosen uniformly at random. Its particle then moves to a neigboring empty vertex or stays where it is, uniformly at random. Equivalently, the chain is lazy with probability $\left \lfloor \frac{n}{k} \right \rfloor \cdot \frac{k-\tilde{k}}{n}$, and otherwise one of the $\tilde{k}$ particles is chosen, and either stays or moves to a neighbor. Since the number of particles $\tilde{k}$ can be upper and lower bounded by constants ($0 \leq \tilde{k} \leq k$), the mixing time within $C$ is independent of $n$. Therefore, we conclude that the overall mixing time is $O(n^2)$.
\end{proof}

\subsection{The Complete Graph Case}
Note that the complete graph case for $\beta < 0$ is equivalent to the vector of proportions chain in the antiferromagnetic Curie--Weiss Potts model.
\begin{theorem}\label{thm:complete-negative}
On the complete graph with $k$ vertices, the mixing time is $O\left(n \log n\right)$ for all $\beta$ satisfying $-\frac{k}{10}< \beta \leq 0$.
\end{theorem}
The proof relies on the following two lemmas.
\begin{lemma}\label{lemma:negative-comparison}
Let $\left(X_t, t \geq 0 \right)$ be the ARW chain for any $\beta < 0$ and let $\left(Y_t, t \geq 0 \right)$ be a chain of independent particles ($\beta = 0$). Set $X_0 = Y_0$. For every vertex $v$ and time $t$, 
$$\left|X_t(v) - \frac{n}{k} \right| \stleq \left|Y_t(v) - \frac{n}{k} \right|.$$
\end{lemma}

For $\lambda \geq 0$, let $C(\lambda) \triangleq \left\{x: \left|x(v) - \frac{n}{k}\right| \leq \lambda n \right\}$.
\begin{lemma}\label{lemma:TV-complete-graph}
On the complete graph, if $y = x - e_i + e_j$ and $x, y \in C(\lambda)$, then $$\left \Vert P_x(v, \cdot) - P_y(v, \cdot) \right \Vert_{\text{TV}} \leq  \frac{-5\beta/n }{2 + (k-2) e^{2 \lambda \beta} }$$
for 
\[n \geq \frac{-3\beta}{\log\left(\frac{5}{4}\right)}.\]
\end{lemma}
The proof of Lemma \ref{lemma:negative-comparison} appears later in this section, and the proof of Lemma \ref{lemma:TV-complete-graph} is deferred to the appendix due to its technical nature.

\begin{proof}[Proof of Theorem \ref{thm:complete-negative}]
We may assume that $n$ is large enough so that 
\[n \geq \frac{-3\beta}{\log\left(\frac{5}{4}\right)} \iff e^{-\nicefrac{3\beta}{n}} \leq \frac{5}{4}.\]

Let $\{Y(v), v \in \mathcal{V}\}$ be a random variable distributed according to the stationary distribution of the $\{Y_t(v), v \in \mathcal{V}, t \geq 0 \}$ chain at stationarity. At stationarity, the vertex occupancies are strongly concentrated around their means. By the Hoeffding Inequality, for every $\lambda > 0$,
$$ \mathbb{P} \left(\left|Y(v) - \frac{n}{k} \right|\geq \lambda n \right) \leq 2 e^{-2\lambda^2 n},$$
for every vertex $v$.

Fix $\epsilon > 0$. We wish to upper bound $t_{\text{mix}}(X, \epsilon)$. Note that the mixing time of the $Y$ chain is $O( n \log n)$. To see this, consider a synchronous coupling. The expected amount of time to select all the particles is $O(n \log n)$, and whenever a particle is selected, it moves to a uniformly random location, which is coupled. Now, for all $\epsilon'$, $T_1 \triangleq t_{\text{mix}}\left(Y, \epsilon' \right) = O(n \log n)$. Therefore at time $T_1$, for every $\lambda > 0$,
$$ \mathbb{P} \left(\left|Y_{T_1}(v) - \frac{n}{k} \right|\geq \lambda n \right) \leq 2 e^{-2\lambda^2 n} + \epsilon',$$
for every vertex $v$. By Lemma \ref{lemma:negative-comparison}, it also holds that for every $\lambda > 0$,
$$ \mathbb{P} \left(\left|X_{T_1}(v) - \frac{n}{k} \right|\geq \lambda n \right) \leq 2 e^{-2\lambda^2 n} + \epsilon'$$
for every vertex $v$. Recall that $C(\lambda) = \left\{x: \left|x(v) - \frac{n}{k}\right| \leq \lambda n \right\}$. Then by the Union Bound,
$$\mathbb{P}\left(X_{T_1} \notin C(\lambda)\right) \leq k\left(2 e^{-2\lambda^2n} + \epsilon' \right), $$
for every $\lambda$ and $v$. We observe that for $n$ large enough, there is always an $\epsilon'$ small enough so that \[k\left(2 e^{-2\lambda^2n} + \epsilon' \right)\leq \frac{\epsilon}{2}.\] Then with probability at least $1 - \nicefrac{\epsilon}{2}$, $X_{T_1}$ belongs to $C(\lambda)$.

Next, we establish that for every $\beta < 0$, there exists $\lambda_{\beta}$ such that (1) once the chain enters $C(\lambda_{\beta})$, it takes exponential time to leave $C(2 \lambda_{\beta})$, with high probability; (2) we can applying path coupling within $C(2 \lambda_{\beta})$. The first claim is due to comparison with the $\beta = 0$ chain, as established above. 

We now demonstrate the required contraction for path coupling within $C(2 \lambda)$. 
Recall that we need to define the edges of the graph $\mathcal{H} = \left(\Omega, \mathcal{E}_{\mathcal{H}} \right)$ and choose a length function on the edges. Let $(x,y) \in \mathcal{E}_{\mathcal{H}}$ if $y = x - e_i + e_j$ for some $i \neq j$, and let $l(x,y) = 1$. Consider any pair of neighboring configurations $x$ and $y$. We employ a synchronous coupling, as in Figure \ref{fig:coupling}. Namely, the ``extra'' particle at vertex $i$ in configuration $x$ is paired to the ``extra'' particle at vertex $j$ in configuration $y$. All other particles are paired by vertex location. When a particle is selected to be moved in the $x$ configuration, the particle that it is paired to in the $y$ configuration is also selected to be moved.

With probability $\frac{n-1}{n}$, one of the $(n-1)$ pairs that has the same vertex location is chosen. Suppose it is located at vertex $v$. We couple the transitions in the two chains according to the coupling achieving the total variation distance $\left \Vert P_x(v, \cdot) - P_y(v, \cdot) \right \Vert_{\text{TV}}$.

By Lemma \ref{lemma:TV-complete-graph}, when one of the $(n-1)$ particles paired by vertex location is chosen, we can couple them so that they move to the same vertex with probability at least 
\[1 - \frac{ \nicefrac{-5\beta}{n} }{2 + (k-2) e^{4 \lambda \beta} }.\] With the remaining probability, the distance increases by at most $2$.

With the remaining $\frac{1}{n}$ probability, the ``extra'' particle is chosen in both chains. The chains can then equalize with probability $1$ because $P_x(i, \cdot) = P_y(j, \cdot)$ on the complete graph. Therefore, we can bound the Wasserstein distance as follows:
\begin{align*}
W_{\rho}^P(x,y) &\leq 1 + 2\frac{ -5\beta/n }{2 + (k-2) e^{4 \lambda \beta} } - \frac{1}{n} = 1 - \frac{1}{n} \left(1 + \frac{10\beta }{2 + (k-2) e^{4 \lambda \beta} }\right)
\end{align*}

Therefore, in order to achieve contraction, it suffices that
\begin{align}
&1 + \frac{ 10\beta }{2 + (k-2) e^{4 \lambda \beta} } > 0 \iff  -10\beta <  2 + (k-2) e^{4 \lambda \beta} \label{eq:beta-bound}
\end{align}
Fix $0 < \delta < 1$, and let $\lambda_{\beta} = \frac{1}{4 \beta} \log(1 - \delta) > 0$. Then substituting $\lambda = \lambda_{\beta}$, we obtain the condition
\begin{align}
-10\beta <  2 + (k-2)(1-\delta) = k + \delta(2-k). \label{eq:condition}
\end{align}
When $\beta \leq 0$ is such that $-10\beta < k$, there exists $\delta >0$ small enough so that the condition \eqref{eq:condition} holds. We conclude that contraction holds for $-\nicefrac{k}{10} < \beta \leq 0$. 

To summarize the argument, we have shown that in time $O(n \log n)$, the chain enters $C(\lambda_{\beta})$. After that, the chain leaves the larger set, $C(2 \lambda_{\beta})$, with exponentially small probability, which can be disregarded. Within $C(2 \lambda_{\beta})$, the Wasserstein distance with respect to the chosen $\mathcal{H}$ and $\rho$ contracts by a factor of $\left(1  - \theta\left( \frac{1}{n}\right) \right)$, so an additional $O\left(n \log n\right)$ steps are sufficient. Therefore, the overall mixing time is $O\left(n \log n\right)$.
\end{proof}


\begin{proof}[Proof of Lemma \ref{lemma:negative-comparison}]
We claim that there exists a coupling of $\{X_t, Y_t\}$ such that for all $v$ and $t$, $\left|X_t(v) - \frac{n}{k} \right| \leq \left|Y_t(v) - \frac{n}{k} \right|$. Let $\tilde{X}_t(v) = \left|X_t(v) - \frac{n}{k} \right|$ and define $\tilde{Y}_t(v)$ similarly. We claim that for all configurations $x$ and vertices $v$, if $x(v) \neq \frac{n}{k}$, then
\begin{align}\mathbb{P}\left(\tilde{X}_{t+1}(v) = \tilde{X}_t(v) + 1 | X_t = x\right) \leq \mathbb{P}\left(\tilde{Y}_{t+1}(v) = \tilde{Y}_t(v) + 1 | Y_t = x\right) \label{eq:cond-1}
\end{align} and
\begin{align}
\mathbb{P}\left(\tilde{X}_{t+1}(v) = \tilde{X}_t(v) - 1 | X_t = x\right) \geq \mathbb{P}\left(\tilde{Y}_{t+1}(v) = \tilde{Y}_t(v) - 1 | Y_t = x\right). \label{eq:cond-2}
\end{align}
If $x(v) = \frac{n}{k}$, then
\begin{align}
\mathbb{P}\left(X_{t+1}(v) = \frac{n}{k}+1 | X_t = x \right)\leq \mathbb{P}\left(Y_{t+1}(v) = \frac{n}{k}+1 | Y_t = x \right) \label{eq:cond-3}
\end{align} and 
\begin{align}
\mathbb{P}\left(X_{t+1}(v) = \frac{n}{k}-1 | X_t = x \right)\leq \mathbb{P}\left(Y_{t+1}(v) = \frac{n}{k}-1 | Y_t = x \right). \label{eq:cond-4}
\end{align}
In other words, the inequalities \eqref{eq:cond-1}-\eqref{eq:cond-4} state that the $X$ chain is less likely to move in the absolute value--increasing direction, and more likely to move in the absolute value--decreasing direction. These inequalities, along with the fact that $X_0 = Y_0$, suffice to prove the lemma.

The transitions for the $Y_t(v)$ process are $+1$ with probability $\left(1 - \frac{Y_t(v)}{n}\right) \frac{1}{k}$, and $-1$ with probability $\frac{Y_t(v)}{n} \frac{k-1}{k}$. With the remaining probability, $Y_{t+1}(v) = Y_t(v)$. Suppose $x(v) \neq \frac{n}{k}$. There are two cases to analyze when $x(v) \neq \frac{n}{k}$:
\begin{enumerate}
    \item $X_t(v) < \frac{n}{k}$. The probability that $X_{t+1}(v) = X_t(v) - 1$ is upper bounded by $\frac{X_t(v)}{n}\frac{k-1}{k}$, because vertex $v$ is a more likely than average destination. In other words, it is harder to lose a particle from vertex $v$ that has fewer than the average number of particles when $\beta < 0$, compared to when $\beta = 0$. Formally,
   $$\frac{X_t(v)}{n}\left(1 - \frac{\f{\left(X_t(v) - 1\right)}}{\f{\left(X_t(v) - 1\right)} + \sum_{w \neq v} \f{X_t(w)}}\right) < \frac{X_t(v)}{n}\left(1 - \frac{1}{k}\right).$$
    For the same reason, the probability that $X_{t+1}(v) = X_t(v) + 1$ is lower bounded by 
    \[\left(1 - \frac{X_t(v)}{n}\right) \frac{1}{k}.\] Therefore, inequalities \eqref{eq:cond-1} and \eqref{eq:cond-2} hold in this case.
    \item $X_t(v) > \frac{n}{k}$. This time, $v$ is a less likely than average destination. The probability that $X_{t+1}(v) = X_t(v) - 1$ is lower bounded by 
    \[\frac{X_t(v)}{n}\frac{k-1}{k}.\] The probability that $X_{t+1}(v) = X_t(v) + 1$ is upper bounded by \[\left(1 - \frac{X_t(v)}{n}\right) \frac{1}{k}.\] Therefore, inequalities \eqref{eq:cond-1} and \eqref{eq:cond-2} hold in this case also.
\end{enumerate}
Finally, suppose $x(v) = \frac{n}{k}$. Then the probability of losing a particle is upper bounded by $\frac{1}{k}\frac{k-1}{k}$, and the probability of gaining a particle is upper bounded by $\frac{k-1}{k}\frac{1}{k}$. Therefore, inequalities \eqref{eq:cond-3} and \eqref{eq:cond-4} hold.

We conclude that such a coupling exists, and therefore the stochastic dominance holds.
\end{proof}

\section{Conclusion} 
In this paper we have introduced a new interacting particle system model. We have shown that for any fixed graph, the mixing time of the Attracting Random Walks Markov chain exhibits phase transition. We have also partially investigated the Repelling Random Walks model, and we conjecture that model is always fast mixing. 
Beyond theoretical results, it is our hope that the model will find practical use. 

\FloatBarrier
\section{Appendix}
\begin{proof}[Proof of Proposition \ref{prop:stationary-probabilities}]
To compute the stationary probabilities $\lambda(r), r \in \{0, 1, \dots, D\}$, note that we can disregard the initial uniform particle choice, and simply consider a Markov chain on a graph with $(D+1)$ nodes as in Figure \ref{fig:Z-chain} or \ref{fig:Z-chain-small}.

When $D = 1$, we have $\lambda(0) = q \lambda(0) + q \lambda(1) \implies \lambda(0) = q$. 

Next, consider $D=2$. We have
\begin{align*}
\lambda(2) &= (1-p) \lambda(2) + (1-q) \lambda(1) \implies \lambda(1) = \frac{p}{1-q} \lambda(2)\\
\lambda(0) &= q \lambda(0) + q \lambda(1) \implies \lambda(0) = \frac{pq}{(1-q)^2} \lambda(2).
\end{align*}
Since $\lambda(0) + \lambda(1) + \lambda(2) = 1$, we have
\begin{align*}
\left( \frac{pq}{(1-q)^2} +  \frac{p}{1-q} + 1 \right) \lambda(2) = 1,
\end{align*}
and so
\begin{align*}
 \lambda(0) &= \frac{ \frac{pq}{(1-q)^2}}{ \frac{pq}{(1-q)^2} +  \frac{p}{1-q} + 1} = \frac{q}{1+\frac{(1-q)^2}{p}}.
\end{align*}
Finally, consider the case $D \geq 3$. We solve the equations for the stationary distribution.
\begin{align}
&\lambda(D) = (1-p) \lambda(D) + (1-p) \lambda(D-1) \nonumber \\
\implies &\lambda(D-1) = \frac{p}{1-p} \lambda(D) \nonumber \\
&\lambda(D-1) = p\lambda(D) + (1-p) \lambda(D-2) \nonumber \\
\implies &\lambda(D-2) = \frac{1}{1-p} \left(\frac{p}{1-p} \lambda(D) - p \lambda (D) \right) = \left(\frac{p}{1-p} \right)^2 \lambda(D) \nonumber \\
&\dots \nonumber\\
&\lambda(D-i) = \left(\frac{p}{1-p} \right)^i \lambda(D) \text{ for } 0 \leq i \leq D - 2 \label{eq:0} \\
&\lambda(2) = p \lambda(3) + (1-q) \lambda(1)  \label{eq:1}\\
&\lambda(0) = q \lambda(1) + q \lambda(0)   \label{eq:3}
\end{align}
Using Equations \eqref{eq:0}-\eqref{eq:3},
\begin{align}
\lambda(1) &= \frac{1}{1-q} \left(\left(\frac{p}{1-p} \right)^{D-2} \lambda(D)- p\left(\frac{p}{1-p} \right)^{D-3} \lambda(D)  \right) \nonumber\\
&= \frac{p}{1-q} \left(\frac{p}{1-p} \right)^{D-2} \lambda(D) \nonumber \\
\lambda(0) &= \frac{pq}{\left(1-q\right)^2} \left(\frac{p}{1-p} \right)^{D-2} \lambda(D).\label{eq:lambda0}
\end{align}
Since $\sum_{i=0}^D \lambda(i) = 1$,
\begin{align*}
&\lambda(D) \left(\frac{pq}{\left(1-q\right)^2} \left(\frac{p}{1-p} \right)^{D-2} + \frac{p}{1-q} \left(\frac{p}{1-p} \right)^{D-2} + \sum_{i=0}^{D-2} \left( \frac{p}{1-p} \right)^i \right) = 1\\
&\lambda(D) = \frac{1}{\frac{p}{\left(1-q\right)^2} \left(\frac{p}{1-p} \right)^{D-2} + \left(\frac{1- \left( \frac{p}{1-p} \right)^{D-1}}{1-\frac{p}{1-p}} \right) }.
\end{align*}
Substituting into \eqref{eq:lambda0}
\begin{align*}
\lambda(0) &= \frac{q}{1 + \frac{\left(1-q \right)^2}{p} \left(\frac{p}{1-p} \right)^{2-D} \left(\frac{1- \left( \frac{p}{1-p} \right)^{D-1}}{1-\frac{p}{1-p}} \right)}. \qedhere
\end{align*}
\end{proof}

\begin{proof}[Proof of Lemma \ref{lemma:TV-complete-graph}]
Let 
\small
\begin{align*}
B(x) &= \sum_{u \not \in \{i,j,v\} } \f{x(u)} + \mathbbm{1}_{v \not \in \{i,j\} } \f{(x(v) -1)}\\
C(x) &= \mathbbm{1}_{v \neq i} \f{x(i)}  + \mathbbm{1}_{v = i} \f{(x(i)-1)} + \mathbbm{1}_{v \neq j} \f{x(j)}  + \mathbbm{1}_{v = j} \f{(x(j)-1)}, \text{ and}\\
D(x) &=  \mathbbm{1}_{v \neq i} \f{(x(i)-1)}  + \mathbbm{1}_{v = i} \f{(x(i)-2)} + \mathbbm{1}_{v \neq j} \f{(x(j)+1)}  + \mathbbm{1}_{v = j} \f{x(j)}.
\end{align*}
\normalsize
Then we can write
\begin{align*}
P_x(v, w) &= \frac{\mathbbm{1}_{v \neq w}\f{x(w)} + \mathbbm{1}_{v = w}\f{(x(w)-1)} }{B(x) + C(x)}
\end{align*}
and
\begin{align*}
P_y(v, w) &= \frac{\mathbbm{1}_{v \neq w}\f{y(w)} + \mathbbm{1}_{v = w}\f{(y(w)-1)} }{B(x) + D(x)}.\end{align*}
To check the sign of $P_x(v,w) - P_y(v,w)$, it is equivalent to check the sign of
\begin{align*}
\f{x(w)} \left(B(x) + D(x) \right) -\f{y(w)} \left(B(x) + C(x)  \right).
\end{align*}
Next we show that for fixed $v$, the sign of $P_x(v, w) - P_x(v, w)$ is the same for all $w \not \in \{i,j\}$. Suppose $w \not \in \{i,j\}$. Then $\f{x(w)} = \f{y(w)}$, and is equivalent to check the sign of the expression $D(x) - C(x)$. Since this expression does not depend on $w$, we conclude that the sign is the same for all $w \not \in \{i,j\}$. 

If $P_x(v, w) - P_x(v, w) \geq 0$ for all $w \not \in \{i,j\}$, then
\begin{align*}
\Vert P_x(v, \cdot) - P_y(v, \cdot) \Vert_{\text{TV}} &= \max \{ P_y(v, i) - P_x(v,i), 0\} + \max \{ P_y(v, j) - P_x(v,j), 0\} .
\end{align*}
Similarly, if $P_x(v, w) - P_x(v, w) < 0$ for all $w \not \in \{i,j\}$, then
\begin{align*}
\Vert P_x(v, \cdot) - P_y(v, \cdot) \Vert_{\text{TV}} &= \max \{ P_x(v, i) - P_y(v,i), 0\} + \max \{ P_x(v, j) - P_y(v,j), 0\} .
\end{align*}
Therefore,
\begin{align*}
\Vert P_x(v, \cdot) - P_y(v, \cdot) \Vert_{\text{TV}} &\leq \left |P_x(v, i) - P_y(v,i) \right| + \left| P_x(v, j) - P_y(v,j) \right|.
\end{align*}
Consider the ratio of denominators of $P_x(v,w)$ and $P_y(v,w)$. We have
\begin{align*}
e^{\nicefrac{\beta}{n}} \leq \frac{B(x) + C(x)}{B(x) + D(x)} \leq e^{-\nicefrac{\beta}{n}}.
\end{align*}
We first bound $\left|P_x(v, i) - P_y(v,i) \right|$. If $v \neq i$, we obtain
\begin{align*}
&\left | P_x(v, i) - P_y(v,i) \right|\\
&\leq \frac{1}{B(x) + C(x)} \max \left\{ \left| \f{x(i)} - \f{(x(i)-1)} e^{\nicefrac{\beta}{n}} \right|,  \left| \f{x(i)} - \f{(x(i)-1)} e^{-\nicefrac{\beta}{n}} \right|\right\}\\
&= \frac{\f{x(i)}}{B(x) + C(x)} \left(e^{-\nicefrac{2\beta}{n}} -1  \right).
\end{align*}
Similarly, if $v = i$, we obtain
\begin{align*}
\left | P_x(v, i) - P_y(v,i) \right| &\leq e^{-\nicefrac{\beta}{n}} \frac{\f{x(i)}}{B(x) + C(x)} \left(e^{-\nicefrac{2\beta}{n}} -1  \right).
\end{align*}
We similarly bound $\left|P_x(v, j) - P_y(v,j) \right|$. If $v \neq j$, we obtain
\begin{align*}
&\left | P_x(v, j) - P_y(v,j) \right| \\
&\leq \frac{1}{B(x) + C(x)} \max \left\{ \left| \f{x(j)} - \f{(x(j)+1)} e^{\nicefrac{\beta}{n}} \right|,  \left| \f{x(j)} - \f{(x(j)+1)} e^{-\nicefrac{\beta}{n}} \right|\right\}\\
&= \frac{\f{x(j)}}{B(x) + C(x)} \left(1 - e^{\nicefrac{2\beta}{n}}  \right).
\end{align*}
If $v = j$, we obtain
\begin{align*}
\left | P_x(v, j) - P_y(v,j) \right| &\leq e^{-\nicefrac{\beta}{n}} \frac{\f{x(j)}}{B(x) + C(x)} \left(1 - e^{\nicefrac{2\beta}{n}}  \right).
\end{align*}
For any choice of $v$,
\[C(x) \geq \f{x(i)}  + \f{x(j)} .\]
Therefore,
\begin{align*}
&\left | P_x(v, i) - P_y(v,i) \right| + \left | P_x(v, j) - P_y(v,j) \right| \\
&\leq \frac{e^{-\nicefrac{\beta}{n}}}{B(x) + \f{x(i)} + \f{x(j)}} \left(\f{x(i)} \left(e^{-\nicefrac{2\beta}{n}} -1\right) + \f{x(j)} \left(1 - e^{\nicefrac{2\beta}{n}}\right) \right)\\
&= \frac{e^{-\nicefrac{\beta}{n}} \left(1 - e^{\nicefrac{2\beta}{n}}\right) }{B(x) + \f{x(i)} + \f{x(j)}} \left(\f{x(i)} e^{-\nicefrac{2\beta}{n}} + \f{x(j)} \right)\\
&\leq \frac{e^{-\nicefrac{3\beta}{n}} \left(1 - e^{\nicefrac{2\beta}{n}}\right) }{B(x) + \f{x(i)} + \f{x(j)}} \left(\f{x(i)}  + \f{x(j)} \right).
\end{align*}
Recall that $x \in C(\lambda)$. We upper bound by setting $x(i)$ and $x(j)$ to their lower bounds, and $x(u)$ to its upper bound for $u \not \in \{i,j\}$.
\begin{align*}
&\left | P_x(v, i) - P_y(v,i)\right| + \left | P_x(v, j) - P_y(v,j) \right| \\
&\leq \frac{e^{-\nicefrac{3\beta}{n}} \left(1 - e^{\nicefrac{2\beta}{n}}\right) }{(k-2) \f{\left(\frac{n}{k} + \lambda n\right)} + 2 \f{\left(\frac{n}{k} - \lambda n\right)}} 2 \f{\left(\frac{n}{k} - \lambda n\right)}\\
&= \frac{2e^{-\nicefrac{3\beta}{n}} \left(1 - e^{\nicefrac{2\beta}{n}}\right) }{(k-2) e^{2\lambda \beta} + 2 } \\
&\leq \frac{2e^{-\nicefrac{3\beta}{n}} \left(-\nicefrac{2\beta}{n}\right) }{(k-2) e^{2\lambda \beta} + 2 } \\
&\leq \frac{-\nicefrac{5\beta}{n} }{(k-2)  e^{2\lambda \beta} + 2 },
\end{align*}
where in the second-last inequality we have used the fact that $1 + z \leq e^z$ and the last inequality holds when $e^{-\nicefrac{3\beta}{n}} \leq \nicefrac{5}{4}$.
\end{proof}




\bibliographystyle{amsplain}    
\bibliography{../bibliography} 

\textbf{Acknoeledgments}: Many thanks to David Gamarnik and Patrick Jaillet for numerous helpful discussions. We appreciate the careful
editing by D. Gamarnik. The work benefited in a pivotal way from discussions with Eyal Lubetzky and Reza Gheissari,
especially in the proof of slow mixing. The idea of using a lower-bounding comparison chain is due to R. Gheissari. J.
Gaudio is grateful to E. Lubetzky for kindly hosting her at NYU. We acknowledge Yuval Peres for several helpful
discussions, including an idea used in the proof of fast mixing for Repelling Random Walks on the complete graph. We thank the anonymous reviewer for very helpful comments
that improved the clarity of this work, including observation~\eqref{eq:srw_strong}. J.
Gaudio was supported by a Microsoft Research PhD Fellowship. Y. Polyanskiy was supported 
in part by the MIT-IBM Watson AI Lab and by the Center for Science of Information (CSoI),
an NSF Science and Technology Center, under grant agreement CCF-09-39370.


\end{document}